\documentclass[final,12pt]{colt2025} % Include author names

\title[Solving Convex-Concave Problems with $\tilde{\mathcal{O}}(\epsilon^{-4/7})$ Second-Order Oracle Complexity]{Solving Convex-Concave Problems with $\tilde{\mathcal{O}}(\epsilon^{-4/7})$ Second-Order \\ Oracle Complexity}
\usepackage{times}
\usepackage{thmtools}
\usepackage{amsmath}
\usepackage{thm-restate}
\usepackage{booktabs}
\usepackage{microtype}
\usepackage{cancel}

\usepackage{amsmath,amsfonts,bm}
\def\eqref#1{equation~\ref{#1}}
\def\1{\bm{1}}

\def\va{{\bm{a}}}
\def\vb{{\bm{b}}}
\def\vc{{\bm{c}}}

\def\vg{{\bm{g}}}

\def\vu{{\bm{u}}}
\def\vv{{\bm{v}}}

\def\vx{{\bm{x}}}
\def\vy{{\bm{y}}}
\def\vz{{\bm{z}}}

\def\mA{{\bm{A}}}

\def\mF{{\bm{F}}}

\DeclareMathAlphabet{\mathsfit}{\encodingdefault}{\sfdefault}{m}{sl}
\SetMathAlphabet{\mathsfit}{bold}{\encodingdefault}{\sfdefault}{bx}{n}

\def\gI{{\mathcal{I}}}

\def\gM{{\mathcal{M}}}

\def\gO{{\mathcal{O}}}

\def\gS{{\mathcal{S}}}

\def\gX{{\mathcal{X}}}
\def\gY{{\mathcal{Y}}}
\def\gZ{{\mathcal{Z}}}

\def\sR{{\mathbb{R}}}

\usepackage{algorithm}
\usepackage{algorithmic}
\usepackage{xcolor}
\usepackage{nicefrac}
\usepackage{cancel}
\newtheorem{thm}{Theorem}[section]
\newtheorem{dfn}{Definition}[section]

\newtheorem{lem}{Lemma}[section]
\newtheorem{asm}{Assumption}[section]
\newtheorem{cor}{Corollary}[section]

\coltauthor{%
 \Name{Lesi Chen} \Email{chenlc23@mails.tsinghua.edu.cn}\\
 \addr IIIS, Tsinghua University\\
       Shanghai Qizhi Institute
 \AND
 \Name{Chengchang Liu} \Email{ccliu22@cse.cuhk.edu.hk} \\
 \addr 
 Department of Computer Science $\&$ Engineering, The Chinese University of Hong Kong 
 \AND 
 \Name{Luo Luo} \Email{luoluo@fudan.edu.cn} \\
 \addr {School of Data Science, Fudan University  \\
 Shanghai Key Laboratory for Contemporary Applied Mathematics}
 \AND 
   \Name{Jingzhao Zhang}${}^\dagger$ 
 \Email{jingzhaoz@mail.tsinghua.edu.cn} \\
 \addr {IIIS, Tsinghua University\\
        Shanghai AI Lab\\
       Shanghai Qizhi Institute
 }
}

\begin{document}

\maketitle
\begingroup
%\begin{NoHyper}
% \renewcommand\thefootnote{${}^*$}
% \footnotetext{Equal contributions.}
% \end{NoHyper}
\begin{NoHyper}
\renewcommand\thefootnote{${}^\dagger$}
\footnotetext{The corresponding author.}
\end{NoHyper}
\endgroup
\begin{abstract}%
% Conventional wisdom in the literature optimization,  backed by the theory
Previous algorithms can solve convex-concave minimax problems $\min_{x \in \mathcal{X}} \max_{y \in \mathcal{Y}} f(x,y)$ with $\gO(\epsilon^{-2/3})$ second-order oracle calls using Newton-type methods.   
This result has been speculated to be optimal because the upper bound is achieved by a natural generalization of the optimal first-order method.
In this work, we show an improved upper bound of
$\tilde{\gO}(\epsilon^{-4/7})$ 
by generalizing the optimal second-order method for convex optimization to solve the convex-concave minimax problem. 
We further apply a similar technique to lazy Hessian algorithms and show that our proposed algorithm can also be seen as a second-order ``Catalyst'' framework \citep{lin2018catalyst} that could accelerate any globally convergent algorithms for solving minimax problems. 
% However, the lower bound arguments assume the algorithms apply a symmetric Newton updates on both variables $x$ and $y$. 
% In this paper, we show a
% a $\tilde{\gO}(\epsilon^{-4/7})$ upper bounds is achievable using asymmetric Newton updates. The gap between symmetric and asymmetric updates does not exist in first-order algorithms; however, in this paper, we find that this phenomenon occurs in second-order algorithms. We also extend our upper bound to the case when lazy Hessian updates are used for reducing computational complexity under the same analysis framework.
% for second-order convex-concave minimax optimization.
% Our method relies on asymmetric Newton updates on variables $x$ and $y$ to break the existing $\Omega(\epsilon^{-2/3})$ lower bound based on symmetric updates. 
\end{abstract}


\section{Introduction}
We study the convex-concave minimax optimization problem over convex and compact sets $\gX$, $\gY$:
\begin{align} \label{prob:main}
\min_{\vx \in \gX} \max_{\vy \in \gY} f(\vx, \vy).
\end{align}
%where $f$ is convex in $\vx$ and concave in $\vy$. 
% where the objective $f(\vx,\vy)$ is convex-concave and the feasible sets  are convex and compact. 
This problem naturally arises from many applications, including finding a Nash equilibrium of a two-player zero-sum game~\citep{von1947theory,karlin2017game,carmon2019variance,carmon2020coordinate,kornowski2024oracle}, solving the Lagrangian function in constrained optimization \citep{ouyang2021lower,kovalev2020optimal,scaman2017optimal}, and many machine learning problems such as adversarial training \citep{zhang2018mitigating}, AUC maximization~\citep{ying2016stochastic} and distributionally robust optimization~\citep{ben2009robust,carmon2022distributionally,curi2020adaptive,song2022coordinate}. 

%The dominating approach to solve Problem (\ref{prob:main}) is based on tools to solve variational inequalities~\citep{kinderlehrer2000introduction}.
Problem (\ref{prob:main}) can also be viewed as variational inequality problems~\citep{kinderlehrer2000introduction}.
Specifically, we let $\vz = (\vx,\vy)$ and $\gZ = \gX \times \gY$, 
then Problem (\ref{prob:main}) can be formulated by the following variational inequality (VI) problem that targets to find a solution $\vz^* \in \gZ $ such that 
\begin{align} \label{prob:VI}
    \langle \mF(\vz), \vz - \vz^* \rangle \ge 0~~\text{for all}~\vz \in \gZ, \ \ 
    \text{where}~~  \mF(\vz) = \begin{bmatrix}
        \nabla_x f(\vx,\vy) \\
        - \nabla_y f(\vx,\vy)
\end{bmatrix}
\end{align}
is monotone. %under the convex-concavity assumption. 
The algorithms for solving monotone VIs are well-studied in the literature and can be used to solve our Problem (\ref{prob:VI}).  
%The algorithms to solve a VI with a monotone operator are well-studied in the literature of optimization, and one can apply these algorithms to solve Problem (\ref{prob:main}).  
The seminal work by 
%\citet{korpelevich1976extragradient} showed the extragradient (EG) method can find an $\epsilon$-solution to Problem (\ref{prob:VI}) in $\gO(\epsilon^{-1})$ calls of the first-order oracle $\mF(\vz)$.
\citet{korpelevich1976extragradient} showed the extragradient (EG) method can find an $\epsilon$-solution to Problem (\ref{prob:main}) with the $\gO(\epsilon^{-1})$ calls of the first-order oracle~$\mF(\vz)$.
\citet{nemirovskij1983problem,ouyang2021lower}
showed that any first-order algorithm for a bilinear minimax problem must make at least $\Omega(\epsilon^{-1})$ first-order queries, which proves the optimality of EG in first-order optimization.
\citet{monteiro2012iteration} proposed the Newton Proximal Extragradient (NPE) method 
as a second-order extension of EG, and they showed that NPE can provably find an $\epsilon$-solution to Problem (\ref{prob:main}) with $\gO(\epsilon^{-2/3})$ calls of the second-order oracle $(f(\vz), \nabla f(\vz), \nabla^2 f(\vz) )$.
The NPE algorithm has been simplified by many subsequent works \citep{adil2022optimal,huang2022approximation,lin2022explicit,bullins2022higher,liu2022regularized,lin2022perseus}, while they still require the $\gO(\epsilon^{-2/3})$ second-order oracle calls, which are speculated to be optimal under restricted conditions (see Appendix~\ref{apx:diss-lower}).

% Recently, \citet{adil2022optimal,lin2022perseus} gave an instance of convex-concave minimax problem such that a class of second-order methods require at least $\Omega(\epsilon^{-2/3})$ oracle calls to find an $\epsilon$-solution.

% The lower bound arguments by \citet{adil2022optimal,lin2022perseus} assumes the algorithm uses \textbf{symmetric} updates in variables $\vx$ and $\vy$, which means the algorithm always perform Newton updates on the aggregated variable $\vz = (\vx,\vy) $. Under their assumptions, the second-order algorithms generate queries as follows:
% \begin{align*}
%     \vz_{k+1} \in \vz_0 + {\rm Span} \left\{ \Gamma_{\vv_j}(\vz_j), 0 \le j \le k  \right\},
% \end{align*}
% where $\Gamma_{\vv_j}(\vz_j)$ denotes the solution set of the following VI problem corresponding to a cubic regularized Newton step on the aggregated variable $\vz = (\vx,\vy) $: 
% {\small 
% \begin{align*}
%     {\rm CRN}_{\vv_j}(\vz_j) : = \left\{ \hat \vz \in \gZ: \langle \mF(\vv_j) + \nabla \mF(\vv_j)(\hat \vz - \vz_j) + \frac{M}{3} \Vert \hat \vz - \vz_j \Vert (\hat \vz- \vz_j), \vz - \hat \vz \rangle \ge 0 , ~\forall \vz \in \gZ \right\},
% \end{align*}}
% where $\vv_{k+1} \in \vz_0 \in {\rm Span} \left\{\mF(\vz_j), 0 \le j \le k \right\}$. See \citep[Assumption 3. 10]{lin2022perseus}. Their assumptions on the algorithm class 
% are highly motivated by the procedure of extrgradient type methods \citep{monteiro2012iteration,bullins2022higher,adil2022optimal,huang2022approximation,lin2022explicit,nesterov2023high} that solving minimax problems via VI approach. 

In this work, we show that it is possible to break the barrier $\Omega(\epsilon^{-2/3})$  
for second-order convex-concave minimax optimization with practical algorithms. 
We propose the Minimax-AIPE algorithm, which generalizes the second-order methods A-NPE \citep{monteiro2013accelerated} for minimization problems.
We prove our method requires at most $\tilde \gO\big(D_x^{6/7} D_y^{6/7} (\rho / \epsilon)^{4/7}\big)$ second-order oracle calls to find an $\epsilon$-solution for Problem (\ref{prob:main}) with the assumption of $\rho$-Lipschitz continuous Hessian, where $D_x$ and~$D_y$ are the diameters of $\gX$ and $\gY$, respectively. 
Our result significantly improves the existing upper bound of $\gO\left( \max\{D_x^2, D_y^2 \} (\rho /\epsilon)^{2/3} \right)$ achieved by the NPE method and its variants.

{Our proposed Minimax-AIPE is a triple-loop algorithm: the outer loop
runs $\tilde \gO \big( D_x^{6/7} (\gamma/ \epsilon)^{2/7} \big)$ iterations of the restarted Accelerated Inexact Proximal Exragradient (AIPE-restart, Algorithm \ref{alg:ANPE-restart}) in variable $\vx$; the middle loop runs $\tilde \gO \big( D_y^{6/7} (\gamma/ \epsilon)^{2/7} \big)$ iterations of AIPE-restart in variable $\vy$; and the inner loop implements a minimax proximal step
\begin{align*}
    \min_{\vx \in \gX} \max_{\vy \in \gY} f(\vx,\vy) + \frac{\gamma}{3} \Vert \vx - \bar \vx \Vert^3 - \frac{\gamma}{3} \Vert \vy - \bar \vy \Vert^3  
\end{align*}
via a linearly convergent method $\gM$, where $(\bar \vx, \bar \vy)$ is the proximal center and $\gamma$ is a hyper-parameter to be chosen. By applying different algorithms $\gM$ to implement the proximal step and tuning the hyper-parameter $\gamma$ accordingly, we naturally obtain the acceleration of these algorithms. For instance, choosing $\gM$ to be the NPE method \citep{monteiro2012iteration} achieves the  
$\tilde \gO\big(D_x^{6/7} D_y^{6/7} (\rho / \epsilon)^{4/7}\big)$ second-order oracle complexity as we claimed. Moreover,
our method is also compatible with the lazy Hessian technique to reduce the computational complexity of second-order methods. Choosing $\gM$ to be the recently proposed lazy version of NPE \citep{chen2025computationally} obtains a $\tilde \gO\big(m + m^{5/7} D_x^{6/7} D_y^{6/7} (\rho / \epsilon)^{4/7}  \big)$ upper bound of the number of iterations when reusing Hessians every $m$ iterations. The main results of this paper are shown in Table \ref{tab:main-result}.
}

% Our results can be intuitively understand as the  complexity of A-IPE in $\vx$ multiplied by the  in $\vy$. 
% % 
% The proposed Algorithm~\ref{alg:ANPE-restart} allows and requires inexact proximal oracles, which can be obatained via existing second-order methods such as NPE \citep{monteiro2012iteration,alves2023search} / MP-2 \citep{bullins2022higher,adil2022optimal} / ARE \citep{huang2022approximation} or
% any other variants such as OGDA-2 \citep{jiang2022generalized,jiang2024adaptive} and Perseus \citep{lin2022perseus}. 
% % Our acceleration framework solves a series of auxiliary regularized convex-concave minimax problems via existing , which can be 

\paragraph{Notations} We use $\Vert \, \cdot \, \Vert$ to denote the Euclidean norm for vectors and the spectral norm for matrices. We hide logarithmic factors in the notation $\tilde \gO(\,\cdot\,)$.
For a set $\gS$, we denote $\gI_{\gS}$ to its indicator function and $\partial \gI_{\gS}$ denotes the subgradient of the function.
We define $D:=\max\{D_{x},D_{y}\}$.
To simplify notations, we let $\vz = (\vx,\vy)$ and define the operator as 
$\mF(\vz) = \begin{bmatrix}
        \nabla_x f(\vx,\vy) \\
        - \nabla_y f(\vx,\vy)
\end{bmatrix}$.

\begin{table*}[t] 
\caption{We compare the theoretical results to find an $\epsilon$-solution for Problem (\ref{prob:main}) of representative second-order methods NPE \citep{monteiro2012iteration} and LEN \citep{chen2025computationally} before and after applying our proposed Minimax-AIPE acceleration framework.
}
\label{tab:main-result}
\centering
\centering
\begin{tabular}{c c c c c}
\hline 
Method   & Before Acc. & After Acc. & \# Hessians  & Reference \\ 
\hline \hline  \addlinespace 
% A-CRN & \cite{nesterov2008accelerating}   & $L$
% & $\gO(d^{\omega} (L / \epsilon)^{\frac{1}{3}})$
% \\ \addlinespace
% \vspace{-0.2cm}
%  & \\ 
NPE  &   $\gO(\epsilon^{-2/3}) $ & $\tilde \gO(\epsilon^{-4/7}) $ & every step &{Theorem \ref{thm:Minimax-AIPE-CC}}
\\  \addlinespace
LEN & $\gO(m+ m^{2/3} \epsilon^{-2/3})$ & $\tilde \gO(m+m^{5/7} \epsilon^{-4/7})$  & once per $m$ steps& {Theorem \ref{thm:Minimax-AIPE-CC-lazy}} \\ \addlinespace
% LEN & Ours & $\gO(m^{2/3} \epsilon^{-2/3})$ &  \\
% Minimax-AIPE & Ours & $\tilde \gO(m^{5/7} \epsilon^{-4/7})$ & once per $m$ steps \\
\hline
    \end{tabular}
\end{table*}

\section{Related Works}
We review the existing results of different methods for convex-concave minimax optimization and the related acceleration techniques used in our methods.

\paragraph{First-Order Methods} The optimal oracle complexity of first-order algorithms to solve convex-concave minimax problems is well known. Various methods, including extragradient \citep{korpelevich1976extragradient}, dual extrapolation \citep{nesterov2007dual}, optimistic gradient descent ascent \citep{popov1980modification,mokhtari2020unified,mokhtari2020convergence}, achieve the upper bound $\gO( \max\{D_x^2,D_y^2 \} / \epsilon)$, which is known to be optimal when $D_x = D_y$ \citep{nemirovskij1983problem}. 
Inspired by the accelerated proximal point algorithm / Catalyst \citep{lin2018catalyst}, 
\citet{lin2020near} proposed the Minimax-APPA algorithm which achieves the upper bound of $\tilde \gO( D_x D_y / \epsilon)$, which fully matches the lower bound when $D_x \ne D_y$ \citep{ouyang2021lower}. Subsequent works proposed enhanced algorithms that further improve the logarithmic factors \citep{yang2020catalyst,kovalev2022first,carmon2022recapp} or give sharper rates under the refined $(L_x, L_{xy}, L_y)$-smoothness condition \citep{wang2020improved,jin2022sharper}.

\paragraph{Second-Order Methods} \citet{monteiro2012iteration} generalized the EG algorithm and proposed the Newton Proximal Extragradient (NPE) method with an $\gO \left(\max\{ D_x^2,D_y^2\}  \epsilon^{-2/3} \right)$ upper bound of second-order oracle calls for monotone variational inequalities. 
Subsequently, 
researchers have proposed different generalizations of the optimal first-order methods and have demonstrated the same theoretical guarantees \citep{bullins2022higher,huang2022approximation,adil2022optimal,lin2022perseus,jiang2024adaptive}. 
Very recently, \citet{chen2025computationally} proposed Lazy Extra Newton (LEN) to further reduce the computational complexity of NPE by using lazy Hessian updates \citep{doikov2023second}.
By analogy with the results of first-order methods, these works conjectured that they have achieved the optimal complexity in $\epsilon$ of the second-order algorithm when $D_x = D_y$.
Furthermore, some works \citep{adil2022optimal,lin2022perseus} have attempted to establish a lower bound of $\Omega(\epsilon^{-2/3})$ for this problem.
However, this paper refutes the possibility of the existence of such a lower bound. 
In Appendix \ref{apx:diss-lower}, we
discuss why the lower bounds established in \citet{adil2022optimal,lin2022perseus} are not applicable to ours. 

% In Appendix \ref{apx:diss-lower}, we
% discuss why the attempts in the analysis of \citet{adil2022optimal,lin2022perseus} actually cannot give an $\Omega(\epsilon^{-2/3})$  lower bound. 

\paragraph{Higher-Order Methods} There are also works that generalized first-order and second-order methods to $p$th-order. Assuming a $p$-th order tensor step can be implemented, state-of-the-art $p$th-order methods achieve the iteration complexity of $\gO(\max\{D_x^2,D_y^2 \} \epsilon^{-2/(p+1)})$ \citep{bullins2022higher,nesterov2023high,lin2022perseus}. 
However, the $p$th-order methods for minimax problems remain ``conceptual'' in the case $p \ge 3$ as people do not know how to implement the $p$-th order tensor step in general. 
Therefore, this work only focuses on the implementable case $p=2$, although we expect our result to be generalizable for all $p$.

\paragraph{Acceleration} Our technique for acceleration is also closely related to the acceleration in convex optimization. 
\citet{monteiro2013accelerated} proposed the Accelerated Hybrid Proximal Extragradient (A-HPE) framework. 
The second-order instance of A-HPE, called Accelerated Newton Proximal Extragradient (A-NPE) achieves the $ \gO(\epsilon^{-2/7} \log \epsilon^{-1})$ iteration complexity. \citet{arjevani2019oracle} showed an $\Omega(\epsilon^{-2/7})$ lower bound for second-order convex optimization. Recently, two independent works \citep{carmon2022optimal,kovalev2022first} proposed novel enhancements of A-NPE to remove the $\gO(\log \epsilon^{-1})$ factor caused by line search for the extrapolation coefficient. 
\citet{chen2025computationally} proposed a more computationally efficient version of the search-free A-NPE \citep{carmon2022optimal} by incorporating the lazy Hessian technique \citep{doikov2023second}.
The A-HPE framework is a powerful tool for acceleration, which has also been used to accelerate quasi-Newton methods \citep{jiang2024accelerated}, tensor methods \citep{jiang2019optimal,bubeck2019near,gasnikov2019near,carmon2022optimal}, optimization with ball oracles \citep{carmon2020acceleration,carmon2021thinking,carmon2024whole,carmon2022distributionally}, and parallel optimization \citep{bubeck2019complexity,carmon2023resqueing}. 

% adapts the hard instance in  \citep{adil2022optimal} and showed the output of any second-order algorithm on a hard instance with 
% $f()$ can not find an $\epsilon$-solution in  $\Omega( (D_x D_y^2)^{3/2} \epsilon^{-2/3})$ iterations. 

% do not actually  

\section{Preliminaries} \label{sec:pre}

In this section, we describe the problem setup and some auxiliary definitions in this paper.
% To simplify the notations, we let $\vz = (\vx,\vy)$ and $\mF(\vz) :=
% \begin{bmatrix}
% \nabla_x f(\vx,\vy) \\
% - \nabla_y f(\vx, \vy) 
%  \end{bmatrix}
% $. Definition of $\partial \gI_{\gZ}$
% \paragraph{Main Assumptions and Convergence Measurement}
We first introduce the standard convex-concave minimax problems and impose the following assumptions for Problem (\ref{prob:main}). 
We first assume that the function is convex-concave and the domain is bounded.%, highlighted below.

\begin{asm}[Convex-concavity] \label{asm:CC}
    We suppose that $f(\,\cdot\, \vy)$ is convex for any fixed $\vy$ and $f(\vx,\,\cdot\,)$ is concave for any fixed $\vx$.
\end{asm}
\begin{asm}[Bounded domain] \label{asm:D}
    We suppose that both the sets $\gX$ and $\gY$ are convex and compact with diameters $D_x:= \sup_{\vx, \vx' \in \gX} \Vert \vx - \vx' \Vert < + \infty$ and  $D_y:= \sup_{\vy, \vy' \in \gY} \Vert \vy - \vy' \Vert < + \infty$.
\end{asm}
We then assume that the function is Lipschitz and smooth. 
\begin{asm}[Lipschitzness] \label{asm:function-lip}
    We suppose $f(\vx,\vy)$ is $L$-Lipschitz continuous for some $L >0$:
    \begin{align*}
        \vert f(\vx,\vy) - f(\vx',\vy') \vert \le L \Vert \vx - \vy \Vert ,\quad \forall \vx,\vx' \in \gX, ~\vy,\vy' \in \gY.
    \end{align*}
\end{asm}

\begin{asm}[Gradient Lipschitzness] \label{asm:grad-lip}
    We suppose the gradient of $f(\vx,\vy)$ is $\ell$-Lipschitz continuous for some $\ell >0$:
    \begin{align*}
        \Vert \nabla f(\vx,\vy) - \nabla f(\vx',\vy') \Vert \le \ell \Vert \vx - \vy \Vert ,\quad \forall \vx,\vx' \in \gX, ~\vy,\vy' \in \gY.
    \end{align*}
\end{asm}

\begin{asm}[Hessian Lipschitzness] \label{asm:Hess-lip}
    We suppose the Hessian of $f(\vx,\vy)$ is $\rho$-Lipschitz continuous for some $\rho >0$:
    \begin{align*}
        \Vert \nabla^2 f(\vx,\vy) - \nabla^2 f(\vx',\vy') \Vert \le \rho \Vert \vx - \vy \Vert ,\quad \forall \vx,\vx' \in \gX, ~\vy,\vy' \in \gY.
    \end{align*}
\end{asm}
We note that previous second-order methods generally do not impose Assumption \ref{asm:function-lip} and \ref{asm:grad-lip}. %However, our algorithm, unlike previously studied ones, utilizes zeroth-order and first-order oracles, and therefore we additional require that Assumptions \ref{asm:function-lip} and \ref{asm:grad-lip} hold.
Although our additionally assumptions look more restricted than existing works, it is important to note that they are mild and can be derived from the higher-order smoothness 
(Assumption \ref{asm:Hess-lip}) in a compact set. 
%Moreover, our results will only have logarithmic dependency on $L$ and $\ell$, and the polynomial part of the complexity depends on $\rho$ as existing second-order methods. 
For this problem, we want to find an approximate solution defined as follows.
\begin{dfn} \label{dfn:eps-sol}
We say $(\hat \vx,\hat \vy) \in \gX \times \gY$ is an $\epsilon$-solution to Problem (\ref{prob:main}) if
\begin{align*}
    {\rm Gap}(\hat \vx, \hat \vy): = \max_{\vy \in \gY} f(\hat \vx, \vy) - \min_{\vx \in \gX} f(\vx, \hat \vy) \le \epsilon.
\end{align*}
\end{dfn}
When $\epsilon=0$, they are the exact saddle points of Problem~(\ref{prob:main}).
We focus on the complexity to find an $\epsilon$-solution. By following previous works, we measure the complexity by the number of cubic regularized Newton (CRN) oracles, which is formally defined as follows.

\paragraph{Second-Order Oracles} We introduce the second-order oracles. We then highlight several important and known lemmas for our analyses.
\begin{dfn} \label{dfn:cubic-VI}
A CRN oracle for Problem (\ref{prob:main}) takes the query point $\bar \vz \in \gZ$ and the regularization parameter $\gamma>0$ as inputs and returns $(\vz,\vu) = {\rm CRN}(\bar \vz, \gamma)$ satisfies:
\begin{align*}
    \langle \mF(\bar \vz) + \nabla \mF(\bar \vz)(\vz - \bar \vz) + \frac{\gamma}{2} \Vert \vz - \bar \vz \Vert (\vz - \bar \vz), \vz' - \vz  \rangle \ge 0, ~\forall \vz' \in \gZ; \\
    \vu =  -\left(\mF(\bar\vz) + \nabla \mF(\bar \vz) (\vz - \bar \vz) + \frac{\gamma}{2} \Vert \vz - \bar \vz \Vert 
    (\vz- \bar \vz)  \right) \in 
    \begin{bmatrix}
        \partial \gI_{\gX} (\vx) \\
        -\partial \gI_{\gY}(\vy)
    \end{bmatrix},
\end{align*}
where $\mF(\vz) = \begin{bmatrix}
        \nabla_x f(\vx,\vy) \\
        - \nabla_y f(\vx,\vy)
\end{bmatrix}$. 
Particularly, 
for minimization problem $\min_{\vz\in\gZ}f(\vz)$ we have
\begin{align*}
    \vz = \arg & \min_{\vz' \in \gZ}  \left\{ 
    f(\vz') + \langle \nabla f(\bar \vz), \vz' - \bar \vz \rangle + \frac{1}{2} \langle \nabla^2 f(\bar \vz) (\vz' - \bar \vz), \vz' - \bar \vz \rangle + \frac{\gamma}{6} \Vert \vz' - \bar \vz \Vert^3
    \right\}; \\
    \vu &= - ( \nabla f(\bar \vz) + \nabla^2 f(\bar \vz) (\vz - \bar \vz) + \frac{\gamma}{2} \Vert \vz - \bar \vz \Vert (\vz-  \bar \vz) ) \in \partial \gI_{\gZ}(\vz).
\end{align*}

% {\small \begin{align*}
%     {\rm CRN}(\bar \vz, \gamma) := \left\{ \vz \in \gZ: \left  \right\}.
% \end{align*}}
\end{dfn}
It is well known that the above oracle can be implemented by transforming it into an equivalent auxiliary one-dimensional problem, which can be efficiently solved with line search procedure \citep[Section 4]{monteiro2012iteration}.

\paragraph{Making Gradient Small}  Our method also relies on tools for making gradients small \citep{allen2018make,foster2019complexity,yoon2021accelerated,chen2024near,lee2021fast,cai2022finite}. We recall the following fact for extragradient (EG) update:
\begin{align} \label{eq:update-eg}
\begin{split}
    \vz_{0.5} &= \arg \min_{\vz \in \gZ} \left\{ \langle \mF(\vz_0),  \vz \rangle + \frac{1}{2 \eta} \Vert \vz-  \vz_0 \Vert^2 \right\}, \\
    \vz_1 &= \arg \min_{\vz \in \gZ} \left\{ \langle \mF(\vz_{0.5}),  \vz \rangle + \frac{1}{2 \eta} \Vert \vz-  \vz_{0} \Vert^2 \right\}
\end{split}
\end{align}

% Our algorithm also utilizes the recent studies on making gradient small in convex minimization \citep{allen2018make,foster2019complexity} and convex-concave minimax problems \citep{yoon2021accelerated,chen2024near,lee2021fast,cai2022finite}. 

\begin{lem}[{\citet[Lemma 12]{cai2022finite}, $T=1$}] \label{lem:eg-grad}
Under Assumption \ref{asm:CC} and \ref{asm:grad-lip},
the extragradient update (\ref{eq:update-eg}) with $\eta \in (0,1/\ell)$ on Problem (\ref{prob:main}) satisfies that
\begin{align*}
    \Vert \mF(\vz_1) + \vc_1 \Vert \le \frac{1 + \eta \ell + (\eta \ell)^2}{\eta \sqrt{1 - (\eta \ell)^2}} \Vert \vz_0 - \vz^* \Vert,
\end{align*}
where $ \vc_1 = (\vz_0 - \vz_1)/\eta - \mF(\vz_{0.5}) \in \partial \gI_\gZ (\vz_1) $.
\end{lem}
For convenience, we also define the above EG step as an oracle.
\begin{dfn} \label{dfn:EG}
Given an input $\vz_0 \in \gZ$ and the operator $\mF$, we let ${\rm EG}(\vz_0, \mF, \eta)$ performs the update as Eq. (\ref{eq:update-eg}) and returns $(\vz_1,\vc_1)$ defined in Lemma \ref{lem:eg-grad}. 
\end{dfn}

% Given a globally convergent second-order method for convex-concave minimax problems, there exist a general restart scheme \citep{huang2022approximation,lin2022perseus} that obtains a linear convergent algorithm for uniformly-convex-uniformly-concave problems. 

% In the following, we present one property of the CRN step that will be useful in our algorithm, which generalizes \citep[Lemma 4.2.5]{nesterov2018lectures} from convex minimization to monotone variational inequalities.
% % The following property of the CRN step will be used in our algorithm.

% % In our algorithm design, we will make use of the following lemma for CRN step.

% \begin{lem} \label{lem:CRN-make-grad-small}
% Under Assumption \ref{asm:Hess-lip}, the CRN oracle $(\vz, \vu) = {\rm CRN}(\bar \vz, 2 \rho)$ ensures
% \begin{align*}
%      \Vert \mF(\vz) + \vu \Vert \le 6 \rho \Vert \bar \vz - \vz^* \Vert^2,
% \end{align*}
% where $\vz^*$ is the solution to Problem (\ref{prob:main}).
% %for some $\vu \in \partial \gI_{\gZ}(\vz)$.
% %where $
% %     $.
% \end{lem}

\paragraph{Uniform Convexity} Finally, we introduce the definition of uniform convexity, which is an important property that will frequently be used in our analysis.
In this paper, we only considered the third-order uniform convexity, which is formally defined as follows.

\begin{dfn} \label{dfn:UC}
    A function $h(\vz): \gZ \rightarrow \sR$ is $\mu$-uniformly convex (of order 3) for some $\mu>0$ if 
    \begin{align*}
        h(\vz) \ge h(\vz') + \langle \nabla h(\vz'), \vz - \vz' \rangle + \frac{\mu}{3} \Vert \vz - \vz' \Vert^3, \quad \forall \vz, \vz' \in \gZ.
    \end{align*}
    We say $h(\vz)$ is $\mu$-uniformly concave if $-h(\vz)$ is $\mu$-uniformly concave.
\end{dfn}
The uniform convexity of order 3 is an important class, where second-order methods such as the cubic regularized Newton method~\citep{nesterov2006cubic} enjoy a linear convergence rate.
In the following, we recall some properties of uniformly convex functions, which is useful to derive the main results of this paper.
% Compared to strong convexity, the uniform convexity (of order 3) has been studied less. However, the former is an important class of functions for second-order algorithms, as these algorithms exhibit linear convergence, which is very important to show the main result of this paper. 
% In the following, we recall some properties of uniformly convex functions.

% Below, we also list some properties of uniformly convex functions.
\begin{lem}[Section 4.2.2, \citet{nesterov2018lectures}] \label{lem:UC-grad-dominant}
Let $h(\vz): \gZ \rightarrow \sR$ be a $\mu$-uniformly convex function and let $\vz^* = \arg \min_{\vz \in \gZ} h(\vz)$. Then for any $\vz \in \gZ$, we have 
\begin{align*}
        \frac{2}{3 \sqrt{\mu}} \Vert \nabla h(\vz) \Vert^{\frac{3}{2}} &\ge h(\vz) - h(\vz^*) \ge \frac{\mu}{3} \Vert \vz - \vz^* \Vert^3.
\end{align*}
\end{lem}
An illustrating example is the cubic function $d(\vz) = (1/3) \Vert \vz \Vert^3$, which has the following properties. 
\begin{lem}[{Lemma 4.2.3 and Lemma 4.2.4, \citet{nesterov2018lectures}}] \label{lem:cubic-func}
    Let $d(\vz) = (1/3) \Vert \vz \Vert^3$ be the cubic function. We have that $d(\vz)$ is $(1/2)$-uniformly convex and has $2$-Lipschitz continuous Hessians.
\end{lem}

\section{The Subroutine: Accelerated Inexact Proximal Extragradient} \label{sec:AIPE}

Before presenting our main algorithm, we first introduce the Accelerated Inexact Proximal Extragradient (AIPE) method for minimizing a convex function $h(\vz): \gZ \rightarrow \sR$.
It will be an important component in our main algorithm later. We present its procedure in Algorithm \ref{alg:ANPE}. It 
generalizes the recently proposed search-free A-HPE method \citep{carmon2022optimal} by allowing the following inexact proximal oracles.

% \begin{algorithm*}[t]  
% \caption{A-NPE$(\vz_0, T, \gamma) $
% }\label{alg:ANPE}
% \begin{algorithmic}[1] 
% \renewcommand{\algorithmicrequire}{ \textbf{Input:}}
% %\REQUIRE Function $f$ and its MS oracle $\gA_{\rm MS}$; Initial $\vz_0$; Iterations $T$; Parameter $\alpha$. \\
% \STATE Set $A_0 = 0$ \\
% \FOR{$t = 0,\cdots,T-1$}
% \STATE Compute a pair $(\lambda_{t+1}, \tilde \vz_{t+1})$ with $\lambda_{t+1}>0$ and $\tilde \vz_{t+1} = {\rm CRN}(\bar \vz_t, \gamma)$ such that 
% \begin{align*}
%     \frac{1}{2} \le \lambda_{t+1} \gamma \Vert \tilde \vz_{t+1} - \bar \vz_t \Vert \le \frac{1}{4},
% \end{align*}
% where
% {\small \begin{align*}
%     \bar \vz_t &= \frac{A_t}{A_t + a_{t+1}} \tilde \vz_t+ \frac{a_{t+1}}{A_t + a_{t+1}} \vv_t,~~ A_{t+1} = A_t + a_{t+1},~~
%     a_{t+1} = \frac{\lambda_{t+1} + \sqrt{\lambda_{t+1}^2 + 4 \lambda_{t+1} A_t}}{2}.
% \end{align*}}
% \\
% \STATE Update $\vv_{t+1} = \vv_t - a_{t+1} \nabla f(\tilde \vz_{t+1})$ 
% \ENDFOR
% \RETURN $\vz_{T}$
% \end{algorithmic}
% \end{algorithm*}

\begin{algorithm*}[t]  
\caption{AIPE$(\vz_0, T, \gamma, \delta) $
}\label{alg:ANPE}
\begin{algorithmic}[1] 
\renewcommand{\algorithmicrequire}{ \textbf{Input:}}
%\REQUIRE Function $f$ and its MS oracle $\gA_{\rm MS}$; Initial $\vz_0$; Iterations $T$; Parameter $\alpha$. \\
\STATE $\vv_0 = \vz_0$, $\bar \vz_0 = \vz_0$, $A_0 = 0$ \\
\STATE $ \tilde \vz_1, \vu_1 = {\color{blue}{\rm iProx}_h(\bar \vz_0, \gamma, \delta)}$ \label{Line:iprox-ini} \\
\STATE $\lambda_1' = \lambda_1 = \gamma \Vert \tilde \vz_1 - \bar \vz_0 \Vert$\\
\FOR{$t = 0,\cdots, T-1$}  
\STATE \quad 
Solve $a'_{t+1} > 0 $ from $ A_t +  a_{t+1}' = {\color{blue}2} \lambda_{t+1}' \left(a_{t+1}'\right)^2$ \label{Line:eq-a}\\
%$a_{t+1'} = \frac{1}{2\lambda_{t+1}'} ( 1 + \sqrt{1 + 4 \lambda_{t+1}' A_t})$ \\
\STATE \quad $A_{t+1}' = A_t + a_{t+1}'$ \\
\STATE \quad $\bar \vz_t = \frac{A_t}{A_{t+1}'} \vz_t + \frac{a_{t+1}'}{A_{t+1}'} \vv_t $ \\
\STATE \quad \textbf{if} $t > 0 $ \textbf{then} 
\STATE \quad  \quad $\tilde \vz_{t+1}, \vu_{t+1}= {\color{blue}{\rm iProx}_h (\bar \vz_t, \gamma, \delta)} $ \label{Line:iprox} \\
\STATE \quad \quad $\lambda_{t+1} = \gamma \Vert \tilde \vz_{t+1} - \bar \vz_t \Vert$ \\
\STATE \quad \textbf{end if} \\
\STATE \quad  \textbf{if} $\lambda_{t+1} \le \lambda_{t+1}'$ \textbf{then} \\
\STATE \quad \quad $a_{t+1} = a_{t+1}' $, $A_{t+1} = A_{t+1}'$ \\
\STATE \quad \quad $\vz_{t+1} = \tilde \vz_{t+1} $ \\
\STATE \quad \quad $\lambda_{t+2}' = \frac{1}{2} \lambda_{t+1}' $\\
\STATE \quad \textbf{else} \\
\STATE \quad \quad $\gamma_{t+1} = \frac{\lambda_{t+1}'}{\lambda_{t+1}}$ \\
\STATE \quad \quad $a_{t+1} = \gamma_{t+1} a_{t+1}'$, $A_{t+1} = A_t + a_{t+1}$ \\
\STATE \quad \quad $ \vz_{t+1} = \frac{(1- \gamma_{t+1}) A_t}{A_{t+1}} \vz_t + \frac{\gamma_{t+1} A_{t+1}'}{A_{t+1}} \tilde \vz_{t+1}$ \\
\STATE \quad \quad $\lambda_{t+2}' = 2 \lambda_{t+1}'$ \\
\STATE \quad \textbf{end if} \\
% \STATE \quad $\vg_{t+1} = {\color{blue} {\rm iGrad} (\tilde \vz_{t+1}, \delta_{t+1})}$ \label{Line:igrad}\\
\STATE \quad {\color{blue}Get $\vg_{t+1}$ such that $\Vert \vg_{t+1} - \nabla h(\tilde \vz_{t+1}) \Vert \le \delta$} \label{Line:igrad} \\
\STATE \quad $\vv_{t+1} = \arg \min_{\vv \in \gZ} \left\{ \langle \vg_{t+1} + \vu_{t+1}, \vv \rangle + \frac{1}{2 a_{t+1}} \Vert \vv - \vv_t \Vert^2 \rangle  \right\}$ \\
\ENDFOR \\
\STATE {\color{blue} $\vz^{\rm out} = \arg \min_{0 \le t \le T}\left\{ \hat h(\vz); \vz \in \{\vz_t, \tilde \vz_t \}\right\}$, where $\vert \hat h(\vz) - h(\vz) \vert \le \delta$.} \label{Line:ifunc}
\RETURN $\vz^{\rm out}$ 
\end{algorithmic}
\end{algorithm*}

\begin{algorithm*}[t]  
\caption{AIPE-restart$(\vz_0, T, \gamma, \delta, S) $
}\label{alg:ANPE-restart}
\begin{algorithmic}[1] 
\renewcommand{\algorithmicrequire}{ \textbf{Input:}}
%\REQUIRE Function $f$ and its MS oracle $\gA_{\rm MS}$; Initial $\vz_0$; Iterations $T$; Parameter $\alpha$. \\
\STATE $\vz^{(0)} = \vz_0$ \\
\FOR{$s=0,\cdots,S-1$}
\STATE \quad $\vz^{(s+1)} =\text{AIPE}(\vz^{(s)}, T, \gamma, \delta) $ \\
\ENDFOR \\
\RETURN $\vz^{(S)}$
\end{algorithmic}
\end{algorithm*}

\begin{dfn}[Inexact Second-Order Proximal Oracle] \label{dfn:inexact-MS-oracle}
An oracle is called a $(\delta,\gamma)$-(second-order)-proximal oracle
for function $h: \sR^d \rightarrow \sR$ if for every $ \bar \vz \in \sR^d$ the points $(\vz, \vu) = {\rm iProx}_h(\bar \vz, \gamma)$ with $\vz \in \gZ$ and $\vu \in \partial \gI_{\gZ}(\vz)$ satisfy
\begin{align*}
    \Vert  \nabla h(\vz) + \vu + \lambda(\vz - \bar \vz) \Vert \le \frac{\lambda}{2} \Vert \vz-  \bar \vz \Vert+ \delta, \quad \lambda = \gamma \Vert \vz - \bar \vz \Vert. 
\end{align*}
\end{dfn}

When $\delta=0$ the above oracle reduces to the oracle used in previous works \citep{monteiro2013accelerated,carmon2022optimal,kovalev2022first,nesterov2021inexact,nesterov2023inexact}, which can be implemented by a CRN oracle.

\begin{lem}[{\citet[Section 3.1]{carmon2022optimal}}] \label{lem:CRN-is-MS}
Assume $h(\vz): \gZ \rightarrow \sR$ has $\rho$-Lipschitz continuous Hessians. The CRN oracle
${\rm CRN}(\,\cdot\,, 2 \rho)$ implements an $( 0, \rho)$-proximal oracle.
\end{lem}

The main modification we made to the algorithm by \citet{carmon2022optimal}  are marked in blue in Algorithm \ref{alg:ANPE}, where Line \ref{Line:iprox} replaces the exact proximal oracle with an inexact one as per Definition~\ref{dfn:inexact-MS-oracle}, and Line \ref{Line:igrad} and Line \ref{Line:ifunc} also allow the following inexact zeroth-order and first-order oracles. 

\begin{dfn}
    We call $\hat h(\vz)$ a $\delta$-zeroth-order oracle of function $h:\gZ \rightarrow \sR$ if $\vert \hat h(\vz)- h(\vz)  \vert \le \delta$.
\end{dfn}

\begin{dfn}
We call $\vg(\vz)$ a $\delta$-first-order oracle of function $h:\gZ \rightarrow \sR$ if $\Vert \vg(\vz)
 - \nabla h(\vz) \Vert \le \delta$.
\end{dfn}

Since the accelerated algorithm is easily affected by errors, it causes challenges in analyzing AIPE.  
In our proof, we restrict the iterations in $t \le T_{\epsilon}$ to avoid some ill-conditioned case after the algorithm has reached an $\epsilon$-solution, where $T_{\epsilon} = \arg \min_{0 \le t \le T} \{ h(\vz); \vz \in \{\vz_t,  \tilde \vz_t \} \}$.  Then we can follow the same steps as \citep[Theorem 1]{carmon2022optimal} to show that it is sufficient to let $\delta \lesssim {\epsilon}/{A_{T_{\epsilon}}}$ to recover the convergence rate as A-HPE \citep{monteiro2013accelerated}. Finally, we provide an upper bound of $A_{T_{\epsilon}}$ to give an explicit parameter setting of $\delta$.
The formal proof is deferred to Appendix~\ref{apx:proof-pre}.
We show the AIPE (Algorithm \ref{alg:ANPE}) can find an $\epsilon$-solution to a convex function in $\gO( (\gamma / \epsilon)^{2/7})$ iterations. Then by incorporating the restart scheme we know Algorithm \ref{alg:ANPE-restart} has a linear convergence rate of $\gO( (\gamma / \mu)^{2/7} \log \epsilon^{-1})$ for minimizing a $\mu$-uniformly convex function, as stated below.

\begin{thm}[AIPE-restart] \label{thm:ANPE-restart}
Assume $h(\vz): \gZ \rightarrow \sR$ is $\mu$-uniformly convex. If $\delta \le {\mu \epsilon^4}/{(144 D^2)}$, then running Algorithm~\ref{alg:ANPE} with
$T = \gO\left( (\gamma/ \mu)^{2/7} \right)$
and $S = \gO(\log (D/\epsilon))$ 
returns $\vz^{(S)}$ such that $\Vert \vz^{(S)} - \vz^* \Vert \le \epsilon$, where $\vz^* = \arg \min_{\vz \in \gZ} h(\vz)$, $D = \sup_{\vz, \vz' \in \gZ}\Vert \vz- \vz' \Vert$.
%Moreover, there exists a line search procedure that implements Line \ref{line:search} in Algorithm \ref{alg:ANPE} with $\gO( {\rm polylog}(\gamma, d_0, 1/(\mu \epsilon^3)) )$ queries of $(\alpha, \delta, \gamma)$-proximal oracles (Definition \ref{dfn:inexact-MS-oracle}) providing that $\alpha \le \frac{1}{128 \gamma^2}$ and $\delta \le \frac{\mu \epsilon^3}{3 \cdot 10^{20} \gamma^2 d_0}$.
\end{thm}

\begin{remark}
The inexact condition $\delta \lesssim {\mu \epsilon^4}/{D^2}$ in Theorem \ref{thm:ANPE-restart} can possibly be refined to $\delta \lesssim {\mu \epsilon^4}/{d_0}$ for $d_0= \Vert \vz_0 - \vz^* \Vert$ by making some additional efforts such as \citep[Lemma 17]{bubeck2019complexity} to show that all iterations of the algorithm lie in a bounded set $\{ \vz\in \sR^d: \Vert \vz - \vz^* \Vert \le \beta \Vert \vz_0 - \vz^*\Vert \}$ for some constant $\beta>0$. But our inexact condition is enough to show the main result in this paper.
\end{remark}

As the $(0,\rho)$-proximal oracle can be implemented by the CRN oracle for a function with $\rho$-Lipschitz continuous Hessians, we can easily obtain the convergence result of the search-free ANPE method \citep{carmon2022optimal} under the restart scheme.

\begin{cor}[ANPE-restart] \label{cor:ANPE}
Assume $h(\vz): \gZ \rightarrow \sR$ is $\mu$-uniformly convex and has $\rho$-Lipschitz continuous Hessians. 
There exists a second-order algorithm that 
returns a point $\vz$ such that $\Vert \vz - \vz^* \Vert \le \epsilon$ in $\gO\left( (\rho/\mu)^{2/7} \log (D/ \epsilon)  \right)$ CRN oracle calls.
\end{cor}
        
\section{Main Result: Achieving the $\tilde \gO(\epsilon^{-4/7})$ Upper Bound} \label{sec:main}

In this section, we introduce our acceleration framework and prove that it can find an $\epsilon$-solution to a convex-concave minimax problems in $\tilde \gO(\epsilon^{-4/7})$ second-order oracle calls. {In Section \ref{subsec:reduction}, we first show that solving the convex-concave minimax problem can be reduced to solving a $\mu_x$-uniformly-convex-$\mu_y$-uniformly-concave minimax problems by adding cubic regularization. Then we show that Algorithm~\ref{alg:Minimax-AIPE} and \ref{alg:Minimax-AIPE-mid} together solve the problem fast if an inexact proximal minimax oracle (\ref{prob:prox-x-y}) is available. We then describe the implementation of the oracle with existing algorithms such as NPE \citep{monteiro2012iteration} and LEN \citep{chen2025computationally} in Algorithm~\ref{alg:Minimax-AIPE-inner}.  } 

% This section is organized as follows. 
% In Section \ref{subsec:reduction}, we reduce the problem to solving  a $\mu_x$-uniformly-convex-$\mu_y$-uniformly-concave minimax problems by adding cubic regularization; 
% In Section \ref{subsec:main}, we present the Minimax-AIPE algorithm and analyze its complexity for uniformly-convex-uniformly-concave problems; 
% In Section \ref{subsec:acc-practice}, we show how to accelerate existing algorithms and compare their convergence rates before and after acceleration.

% In this section, we introduce $\tilde \gO\left( (\rho D_x^3 / \epsilon)^{2/7}(\rho D_x^3 / \epsilon)^{2/7} \right)$ second-order oracle calls. In subsection \ref{subsec:reduction}, we reduce the problem to solving In subsection \ref{subsec:main}, we show a $\tilde \gO\left( (\rho D_x^3 / \mu_x)^{2/7}(\rho D_x^3 / \mu_x)^{2/7} \right)$ upper bound 

\subsection{Reducing to a Uniformly-Convex-Uniformly-Concave Problem} \label{subsec:reduction}

First of all, we apply the regularization trick to make the function uniformly-convex-uniformly-concave. The following lemma connects the approximate solution to the regularized problem and the original problem.
\begin{lem} \label{lem:reduction-UCUC}
Let $\tilde f(\vx,\vy): = f(\vx,\vy) + \frac{\mu_x}{3} \Vert \vx - \vx_0 \Vert^3 - \frac{\mu_y}{3} \Vert \vy - \vy_0 \Vert^3 $. If $(\hat \vx, \hat \vy)$ is an $(\epsilon/3)$-solution to $\tilde f(\vx,\vy)$ with $ \mu_x = \epsilon / (2D_x^3) $ and $\mu_y = \epsilon / (2 D_y^3)$, then it is an $\epsilon$-solution to Problem~(\ref{prob:main}).
\end{lem}

After regularization, the objective becomes uniformly-convex-uniformly-concave. Therefore, we can reduce the problem to study the oracle complexity of second-order algorithms in finding an $\epsilon$-solution to a function that satisfies the following assumption.

\begin{asm} \label{asm:UC-UC}
We suppose $f(\vx,\vy)$ is $\mu_x$-uniformly-convex-$\mu_y$-uniformly-concave, \textit{i.e.}, $f(\vx,\,\cdot\,)$ is $\mu_x$-uniformly convex for any fixed $\vx \in \sR^{d_x}$ and $f(\,\cdot\,\vy)$ is $\mu_y$-uniformly concave for any fixed $\vy \in \sR^{d_y}$, where $\mu_x, \mu_y >0$.
We say $f(\vx,\vy)$ is convex-concave when $\mu_x  = \mu_y =0$.
% \begin{align*}
%     f(\vx, \vy) &\ge f(\vx',\vy) + \langle \nabla_x f(\vx',\vy), \vx- \vy \rangle+ \frac{\mu_x}{3} \Vert \vx - \vx' \Vert^3, \quad \forall \vx, \vx' \in \sR^{d_x}, \vy \in \sR^{d_y}; \\
%      f(\vx, \vy) &\le f(\vx,\vy') + \langle \nabla_y f(\vx,\vy'), \vy- \vy' \rangle-  \frac{\mu_y}{3} \Vert \vy - \vy' \Vert^3, \quad \forall \vx \in \sR^{d_x}, \vy, \vy' \in \sR^{d_y}. 
% \end{align*}
% When $\mu_x =\mu_y=0$, 
\end{asm}
% Our main results would hold for convex-concave functions, \textit{i.e.} the case when $\mu_x = \mu_y = 0$. But to prove the desired result, we will need to add small regularization to the original function to make it uniformly-convex-uniformly-concave. This regularization trick ensures some desired stability of the problem. 
% The following lemma shows we can reduce the problem of finding an $\epsilon$-solution to a convex-convex function to the problem of finding an $\epsilon/2$-solution to a function that satisfies Assumption \ref{asm:UC-UC} by adding $\gO(\epsilon)$-regularization.

Below, we present some useful lemmas which can be derived from Assumption \ref{asm:UC-UC}. 
\begin{lem} \label{lem:max-preserve-convex}
Consider a function $f(\vx,\vy)$ that satisfies Assumption \ref{asm:UC-UC}, then the primal function $\Phi(\vx) = \max_{\vy \in \gY} f(\vx,\vy) $ is $\mu_x$-uniformly convex, and the dual function $\Psi(\vy) = \min_{\vx \in \gX} f(\vx,\vy)$ is $\mu_y$-uniformly concave.
\end{lem}

\begin{lem} \label{lem:cont-sol-set}
Consider a function $f(\vx,\vy)$ that satisfies Assumption  \ref{asm:grad-lip} and \ref{asm:UC-UC}, let $\vx^*(\vy) = \arg \min_{\vx \in \gX} f(\vx,\vy)$ and  $\vy^*(\vx) = \arg \max_{\vy \in \gY} f(\vx,\vy)$. Then both the mappings $\vx^*(\vy)$ and  $\vy^*(\vx)$ are continuous. Furthermore, we have that
\begin{align*}
    \Vert \vy^*(\vx_1)- \vy^*(\vx_2) \Vert^2 &\le \frac{\ell}{\mu_y} \Vert \vx_1 -\vx_2 \Vert, \quad \forall \vx_1,\vx_2 \in \gX; \\
    \Vert \vx^*(\vy_1)- \vx^*(\vy_2) \Vert^2 &\le \frac{\ell}{\mu_x} \Vert \vy_1 - \vy_2 \Vert, \quad \forall \vy_1, \vy_2 \in \gY.
\end{align*}
\end{lem}

Lemma \ref{lem:cont-sol-set} is crucial in our analysis but is not necessary for standard NPE analysis \citep{monteiro2012iteration}. 
For this reason, we additionally require the $\ell$-smoothness in our algorithm, while NPE or its variants may not need. But our additional assumption is mild since the $\ell$-smoothness always hold in a compact set if the function has Lipschitz continuous Hessians.

% In the above, Lemma \ref{lem:max-preserve-convex}
% In this section, we show the $\tilde \gO(\epsilon^{-4/7})$ upper bound by applying the ANPE method on both variable $\vx$ and $\vy$ in minimax optimization . Recall the reduction from Lemma \ref{lem:reduction-UCUC}, we know that it suffices to give a $\tilde{\gO}( \mu_x^{-2/7} \mu_y^{-2/7} )
% $ upper bound for a $\mu_x$-uniformly-convex-$\mu_y$-uniformly-concave objective. 
% Before we present our formal proof, we give a brief proof sketch for better understanding.

\begin{algorithm*}[htbp]  
\caption{Minimax-AIPE}\label{alg:Minimax-AIPE}
\begin{algorithmic}[1] 
\renewcommand{\algorithmicrequire}{ \textbf{Input:}}
\STATE  Run AIPE-restart (Algorithm~\ref{alg:ANPE-restart}) with proximal oracle given by Algorithm~\ref{alg:Minimax-AIPE-mid} to solve 
\begin{align*}
\min_{\vx \in \gX} \Phi(\vx)     
\end{align*}
for finding $\hat\vx$ such that $\Vert\hat\vx - \vx^* \Vert \le \zeta_1$, where $\vx^* = \arg \min_{\vx \in \gX} \Phi(\vx)$.
% \begin{align*}
%      {\rm where} \quad 
%     \end{align*}
%     \\ \vspace{-0.2cm}
%such that $\Vert \vx - \vx^* \Vert \le \zeta_1$.
% where the proximal oracle is obtained via Algorithm \ref{alg:Minimax-AIPE-mid}, $(\vx^*,\vy^*) = \arg \min_{\gX} \max_{\gY} f(\vx,\vy)$.
\STATE Run $\gM_{\rm min}$ to solve 
\begin{align*}
    \max_{\vy \in \gY }f(\hat\vx,\,\cdot\,)
\end{align*}
for finding $\hat \vy$ such that $\Vert \hat\vy - \vy^*(\hat\vx) \Vert \le \zeta_1$, where $\vy^*(\hat\vx) = \arg \max_{\vy \in \gY} f(\hat\vx,\vy)$.
% \begin{align*}
%     , \quad {\rm where} \quad .
% \end{align*}
% \\ \vspace{-0.2cm}
% \STATE Fix $\vx$ and run ANPE-restart on to get  $\vy$ such that $, where $$. \\
\STATE $(\vz^{\rm out}, \vc^{\rm out}) \leftarrow {\rm EG}(\hat \vz, \mF, 1/ (2\ell) )$
% \STATE $((\vx^{\rm out},\vy^{\rm out}), (\vu, \vv)) \leftarrow {\rm CRN}(\hat\vx,\hat\vy, 2 \rho)$ on $f(\vx,\vy)$.
\RETURN $\vz^{\rm out} = (\vx^{\rm out},\vy^{\rm out})$
%\REQUIRE Function $f$ and its MS oracle $\gA_{\rm MS}$; Initial $\vz_0$; Iterations $T$; Parameter $\alpha$. \\
% \STATE $\vz_{1/2} = {\rm CRN}(  \vz_0, \rho)$, \quad $\eta = \frac{1}{\rho \Vert \vz_0 - \vz_{1/2} \Vert}$ \\
% \STATE $\vz_1 = \arg \min_{\vz \in \gZ} \left\{ \langle \eta \mF(\vz_{1/2}), \vz - \vz_0 \rangle + \frac{1}{2} \Vert \vz - \vz_0 \Vert^2  \right\} $.
%\RETURN $\vz_1$ 
\end{algorithmic}
\end{algorithm*}

\subsection{Minimax-AIPE and its Convergence Analysis} \label{subsec:main}

% Motivated by the A-NPE method \citep{monteiro2013accelerated,carmon2022optimal,kovalev2022first} which can achieve the $\tilde\gO\left((\rho/ \mu)^{2/7} \right)$ second-order complexity for minimizing a $\mu$-uniformly convex function (Corollary \ref{cor:ANPE}), we propose the   to achieve the $\tilde\gO\left((\rho/ \mu_x)^{2/7} (\rho/ \mu_y)^{2/7} \right)$ complexity 
We propose the Minimax-AIPE in Algorithm \ref{alg:Minimax-AIPE} for $\mu_x$-strongly-concave-$\mu_y$-strongly-concave minimax problems.
Our algorithm is a general scheme to accelerate second-order minimax optimization. It can be applied on any linear convergent algorithms for uniformly convex minimization problems and uniformly-convex-uniformly-concave minimax problems. We denote such algorithms as $\gM_{\rm min}$ and $\gM_{\rm saddle}$, respectively. 
We assume that $\gM_{\rm min}$ and $\gM_{\rm saddle}$ have the following theoretical guarantee, which is generic and be satisfied by many existing algorithms.

\begin{asm} \label{asm:M-min}
Let $h(\vz): \gZ \rightarrow \sR$ be $\mu$-uniformly convex and has $\rho$-Lipschitz continuous Hessians. We assume $\gM_{\rm min}$ can find a point $\vz$ such that $\Vert \vz - \vz^* \Vert \le \zeta$ in $T_{\rm min}(\rho,\mu) \log (d_0/\zeta) $ iterations, where $d_0= \Vert \vz_0 - \vz^* \Vert$ and $\vz^* = \arg \min_{\vz \in \gZ} h(\vz)$. 
\end{asm}

\begin{asm} \label{asm:M-minimax}
Let $f(\vx,\vy) : \gX \times \gY \rightarrow \sR$ satisfies Assumption \ref{asm:Hess-lip} and \ref{asm:UC-UC}. We assume $\gM_{\rm saddle}$ can find a point $\vz$ such that $\Vert \vz - \vz^* \Vert \le \zeta$ in $T_{\rm saddle}(\rho,\mu) \log (d_0/\zeta) $ iterations, where $d_0= \Vert \vz_0 - \vz^* \Vert$ and $\vz^* = (\vx^*, \vy^*) = \arg \min_{\vx \in \gX} \max_{\vy \in \gY} f(\vx, \vy)$. 
\end{asm}

The Minimax-AIPE is a triple-loop algorithm. 
Below, we introduce the procedures of each loop one by one. The outer loop (Algorithm \ref{alg:Minimax-AIPE}) applies AIPE-restart (Algorithm \ref{alg:ANPE-restart}) to minimize the primal objective $\Phi(\vx):= \max_{\vy \in \gY } f(\vx,\vy)$, which requires the inexact zeroth-order oracles, first-order oracles, and second-order proximal oracles of $\Phi(\vx)$. As both the inexact zeroth-order and first-order oracle of $\Phi(\vx)$ are easily obtainable (see Theorem \ref{thm:get-zo-fo}), the non-trivial one is the proximal oracle, which will be implemented by the middle loop of Minimax-AIPE (Algorithm \ref{alg:Minimax-AIPE-mid}).
If the middle loop can successfully return a proximal oracle, then the convergence of AIPE simply follows
Theorem \ref{thm:ANPE-restart}. 
Below, we show the required precision $\zeta_1$ for Algorithm \ref{alg:Minimax-AIPE} to ensure an $\epsilon$-solution to Problem (\ref{prob:main}).

\begin{algorithm*}[htbp]  
\caption{${\rm iProx}_{\Phi} (\bar \vx, \gamma)$
}\label{alg:Minimax-AIPE-mid}
\begin{algorithmic}[1] 
\renewcommand{\algorithmicrequire}{ \textbf{Input:}}
\STATE Run AIPE-restart (Algorithm~\ref{alg:ANPE-restart}) with proximal oracle given by Algorithm \ref{alg:Minimax-AIPE-inner} to solve 
\begin{align*}
    \max_{\vy \in \gY} \Psi(\vy; \bar \vx)
\end{align*}
for finding $\hat \vy$ such that $\Vert \hat \vy - \vy^*(\bar \vx) \Vert \le \zeta_2$, where $\vy^*(\bar \vx) = \arg \max_{\vy \in \gY} \Psi(\vy; \bar \vx)$.
% \begin{align*}
%     , \quad {\rm where} \quad . 
%     \end{align*}
% \\
% \vspace{-0.2cm}
\STATE Run $\gM_{\rm min}$ to solve 
\begin{align*}
    \min_{\vx \in \gX} g(\,\cdot\, ,  \hat \vy ; \bar \vx)
\end{align*}
for finding $\hat \vx$ such that $\Vert \hat \vx - \vx^*(\hat \vy; \bar \vx) \Vert \le \zeta_2$, where $\vx^*(\hat \vy; \bar \vx)= \arg \min_{\vx \in \gX} g(\vx,\hat \vy; \bar\vx)$.
% \begin{align*}
%      \quad {\rm where} \quad 
% \end{align*}
% \\ \vspace{-0.2cm}
\STATE $(\vx^{\rm out},\vu^{\rm out}) \leftarrow {\rm EG}(\hat \vx, \nabla_x g(\vx, \hat \vy; \bar \vx),  1/ (2(\ell + 2\gamma D)))$. \\
% \STATE Fix $\vy$ and run ANPE-restart with CRN oracles on $g(\,\cdot\,,\vy; \bar \vx)$ to get a point $\vx \in \gX$ such that $$. \\

% \STATE Fix $\vx$ and run ANPE-restart with CRN oracles on $f(\vx,\,\cdot,)$ to get $\vy'$ such that $\Vert \vy' - \vy^*(\vx) \Vert \le \zeta_2$, where $\vy^*(\vx) = \arg \max_{\vy \in \gY} f(\vx,\vy)$. Let $\vg = \nabla_x f(\vx,\vy')$. \\
\RETURN  $(\vx^{\rm out}, \vu^{\rm out})$
%\REQUIRE Function $f$ and its MS oracle $\gA_{\rm MS}$; Initial $\vz_0$; Iterations $T$; Parameter $\alpha$. \\
% \STATE $\vz_{1/2} = {\rm CRN}(  \vz_0, \rho)$, \quad $\eta = \frac{1}{\rho \Vert \vz_0 - \vz_{1/2} \Vert}$ \\
% \STATE $\vz_1 = \arg \min_{\vz \in \gZ} \left\{ \langle \eta \mF(\vz_{1/2}), \vz - \vz_0 \rangle + \frac{1}{2} \Vert \vz - \vz_0 \Vert^2  \right\} $.
%\RETURN $\vz_1$ 
\end{algorithmic}
\end{algorithm*}

\begin{algorithm*}[htbp]  
\caption{${\rm iProx}_{\Psi(\,\cdot\,; \bar \vx)} (\bar \vy, \gamma)$
}\label{alg:Minimax-AIPE-inner}
\begin{algorithmic}[1] 
\renewcommand{\algorithmicrequire}{ \textbf{Input:}}
\STATE Run $\gM_{\rm saddle}$ to solve
\begin{align*}
    \min_{\vx \in \gX} \max_{\vy \in \gY}h(\vx,\vy; \bar \vx, \bar \vy)
\end{align*}
for finding $\hat \vz$ such that $\Vert \hat \vz - \vz^*(\bar \vz) \Vert \le \zeta_3$, where $\vz^*(\bar \vz) = \arg \min_{\vx \in \gX} \max_{\vy \in \gY} h(\vx, \vy; \bar \vx, \bar \vy)$.
% \begin{align*}
%     , \quad {\rm where} \quad .
% \end{align*}
% \vspace{-0.2cm} \\
\STATE $(\vz^{\rm out}, \vc^{\rm out}) \leftarrow \text{EG}(\hat \vz, \mF, 1/(2(\ell +2 \gamma D))$.  
\STATE $(\vx^{\rm out},\vy^{\rm out}) = \vz^{\rm out}$, ~$ (\vu^{\rm out}, \vv^{\rm out}) = \vc^{\rm out}$ 
% \STATE Obtain $$
% Run NPE-restart (Algorithm \ref{alg:NPE-restart})  to get  $(\vx,\vy)$ such that $ \Vert (\vx,\vy) - (\vx^*(\bar \vx,\bar \vy) ,\vy^*(\bar \vx, \bar \vy)) \Vert \le \zeta_3$, $(\vx^*(\bar \vx,\bar \vy) ,\vy^*(\bar \vx, \bar \vy)) =\arg \min_{\vx \in \gX} \max_{\vy \in \gY} h(\vx,\vy; \bar \vx, \bar \vy)$. \\
% \STATE Fix $\vy$ and run ANPE-restart with CRN oracles on $g(\,\cdot\, \vy; \bar\vx)$ to get $\vx'$ such that $\Vert \vx' - \vx^*(\vy; \bar \vx) \Vert \le \zeta_3$, where $\vx^*(\vy; \bar \vx) = \arg \min_{\vx \in \gX} g(\vx,\vy; \bar \vx)$. Let $\vg = \nabla_y g(\vx', \vy; \bar \vx)$.
\RETURN $(\vy^{\rm out}, \vv^{\rm out})$ 
%\REQUIRE Function $f$ and its MS oracle $\gA_{\rm MS}$; Initial $\vz_0$; Iterations $T$; Parameter $\alpha$. \\
% \STATE $\vz_{1/2} = {\rm CRN}(  \vz_0, \rho)$, \quad $\eta = \frac{1}{\rho \Vert \vz_0 - \vz_{1/2} \Vert}$ \\
% \STATE $\vz_1 = \arg \min_{\vz \in \gZ} \left\{ \langle \eta \mF(\vz_{1/2}), \vz - \vz_0 \rangle + \frac{1}{2} \Vert \vz - \vz_0 \Vert^2  \right\} $.
%\RETURN $\vz_1$ 
\end{algorithmic}
\end{algorithm*}

\begin{asm} \label{asm:prox-Phi}
Let $\zeta_1$ be the precision in Algorithm \ref{alg:Minimax-AIPE}. We assume that Algorithm \ref{alg:Minimax-AIPE-mid} can return a $(\delta,\gamma)$-proximal oracle of the primal objective $\Phi(\vx)$ with $\delta \le {\mu_x \zeta_1^4}/{(144 D^2)}$,  where $\Phi(\vx) := \max_{\vy \in \gY} f(\vx,\vy)$ and $D = \max\{ D_x, D_y\}$.
\end{asm}

\begin{thm}[Outer-Loop Complexity] \label{thm:outer}
Let $\zeta_1 \le {\mu_y \epsilon^2}/{(147 \ell^3 D^2)}$.
Under Assumption \ref{asm:D}, \ref{asm:grad-lip}, \ref{asm:Hess-lip}, \ref{asm:UC-UC},\ref{asm:M-min}, and \ref{asm:prox-Phi}, Algorithm \ref{alg:Minimax-AIPE} can find an $\epsilon$-solution to Problem (\ref{prob:main}) in $\gO\left((\gamma / \mu_x)^{2/7} \log (D/ \zeta_1) \right)$ calls of Algorithm~\ref{alg:Minimax-AIPE-mid}, and $ \gO \left( T_{\rm min}(\rho, \mu_y) \log (D / \zeta_1)  \right)$ iterations of $\gM_{\rm min}$.
\end{thm}
To introduce the middle loop of Minimax-AIPE, we denote several surrogate functions:
\begin{align}\label{eq:phi}
\begin{split}
     &\Phi(\vx) :=\max_{\vy \in \gY} f(\vx,\vy),\quad g(\vx,\vy ; \bar \vx):= f(\vx,\vy) + \frac{\gamma}{3} \Vert \vx - \bar \vx \Vert^3, \\
    &\Psi(\vy; \bar \vx) :=  \min_{\vx \in \gX} g(\vx,\vy; \bar \vx)
\end{split}
% \end{align}
\end{align}
The task of the middle loop of Minimax-AIPE
(Algorithm \ref{alg:Minimax-AIPE-mid}) is to
implement a $(\delta, \gamma)$-proximal oracle for the primal objective $\Phi(\vx)$
such that $\delta \lesssim {\mu_x \zeta_1^3}/{D}$ as required by Theorem \ref{thm:outer}. Note that
\begin{align*} 
    {\rm Prox}_{\Phi}(\bar \vx, \gamma) &=
    \arg \min_{\vx \in \gX} \left\{ \Phi(\vx) + \frac{\gamma}{3} \Vert \vx - \bar \vx \Vert^3 \right\}
\end{align*}
can be obtained by solving the equivalent subproblem
\begin{align} \label{prob:prox-x}
    \min_{\vx \in \gX} \max_{\vy \in \gY} g(\vx,\vy; \bar \vx) 
    =  \max_{\vy \in \gY} \Psi(\vy ; \bar \vx).
\end{align}
By Equation (\ref{prob:prox-x}),
Algorithm \ref{alg:Minimax-AIPE-mid} applies AIPE-restart (Algorithm \ref{alg:ANPE-restart}) to maximize $\Psi(\vy; \bar \vx)$, whose proximal oracle is obtained by Algorithm \ref{alg:Minimax-AIPE-inner}.
If Algorithm \ref{alg:Minimax-AIPE-inner} can  achieve the goal, then the
following theorem shows that Algorithm \ref{alg:Minimax-AIPE-mid} can successfully return desired oracles for $\Phi(\vx)$.

\begin{asm} \label{asm:prox-Psi}
Let $\zeta_2$ be the precision in Algorithm \ref{alg:Minimax-AIPE-mid}. Assume that Algorithm \ref{alg:Minimax-AIPE-inner} returns a $(\delta,\gamma)$-proximal oracle of the dual objective $\Psi(\vy; \bar \vx)$ with $\delta \le {\mu_y \zeta_2^4}/{(144 D^2)}$,  where $\Psi(\vy; \bar \vx) := \min_{\vx \in \gX} g(\vx,\vy; \bar \vx)$ and $D = \max\{ D_x, D_y\}$. 
\end{asm}

\begin{thm}[Middle-Loop Complexity] \label{thm:middle}
Let $\zeta_2 = \Omega(1/ {\rm poly} (\rho, \ell,L, D,\gamma, \mu_x^{-1}, \mu_y^{-1}, \zeta_1^{-1} ))$. Under Assumption \ref{asm:D}, \ref{asm:function-lip}, \ref{asm:grad-lip}, \ref{asm:Hess-lip}, \ref{asm:UC-UC}, \ref{asm:M-min}, and \ref{asm:prox-Psi},
Algorithm \ref{alg:Minimax-AIPE-mid} returns a $(\delta,\gamma)$-proximal oracle for $\Phi(\vx)$ that satisfies Assumption \ref{asm:prox-Phi} ($\delta \lesssim {\mu_x \zeta_1^4}/{D^2}$) in $\gO \left( 
(\gamma/\mu_y)^{2/7} \log (D/ \zeta_2)
\right)$ calls of Algorithm \ref{alg:Minimax-AIPE-inner}, and $T_{\rm min}(\rho+ 2 \gamma, \gamma/2) \log (D/ \zeta_2)$ iterations of $\gM_{\rm min}$. We can also obtain the $\delta$- zeroth-order and first-order oracles for $\Phi(\vx)$ in $T_{\rm min}(\rho, \mu_y) \log (D/ \zeta_2)$ iterations of $\gM_{\rm min}$.
% Under Assumption \ref{asm:UC-UC}, \ref{asm:Hess-lip} and \ref{asm:grad-lip}, Algorithm \ref{alg:Minimax-AIPE-mid} can implement a $(\delta, \gamma)$-proximal oracle for function $P(\vx)$ such that $\delta \le \frac{\mu_x \zeta_1^3}{320 D}$ if we set 
% \begin{align*}
% \zeta_2 = \Omega(1/ {\rm poly} (\rho, \ell, D,\gamma, \mu_x^{-1}, \mu_y^{-1}, \zeta_1^{-1} )).
%\end{align*}
\end{thm}
Finally, the inner loop of Minimax-AIPE (Algorithm \ref{alg:Minimax-AIPE-inner}) implements a $(\delta,\gamma)$-proximal oracle for the dual objective $\Psi(\vy; \bar\vx) $. Note that
\begin{align*} 
    {\rm Prox}_{\Psi(\,\cdot\,; \bar \vx)} (\bar \vy, \gamma) &= 
    \arg \max_{\vy \in \gY} \left\{\Psi(\vy; \bar \vx) - \frac{\gamma}{3} \Vert \vy - \bar \vy \Vert^3  \right\}
    % &= \arg \min_{\vx \in \gX} \max_{\vy \in \gY} \{ h(\vx,\vy; \bar \vx,\bar \vy):= f(\vx,\vy) + \frac{\gamma}{3} \Vert \vx - \bar \vx \Vert^3 - \frac{\gamma}{3} \Vert \vy - \bar \vy \Vert^3\}, \\
\end{align*}
can be obtained by solving the equivalent subproblem
\begin{align} \label{prob:prox-x-y}
    \min_{\vx \in \gX} \max_{\vy \in \gY} \left\{ h(\vx,\vy; \bar \vx, \bar \vy): =  f(\vx,\vy) + \frac{\gamma}{3} \Vert \vx - \bar \vx \Vert^3 - \frac{\gamma}{3} \Vert \vy - \bar \vy \Vert^3 \right\}.
\end{align}
% where
% \begin{align*}
%     h(\vx,\vy; \bar \vx,\bar \vy) & :=. 
% \end{align*}
It means we can apply a globally convergent algorithm $\gM_{\rm saddle}$ to find the saddle point of function $h(\vx,\vy; \bar \vx, \bar \vy)$ in Algorithm \ref{alg:Minimax-AIPE-inner}. 
After the above steps, we reduce the optimization of a $\mu_x$-uniformly-convex-$\mu_y$-uniformly-concave function (\ref{prob:main}) to 
the optimization of a $(\gamma/2)$-uniformly-convex-$(\gamma/2)$-uniformly-concave function (\ref{prob:prox-x-y}). Since the latter has a better condition number, the new subproblem is significantly easier to optimize compared to the original problem. 
The following theorem states the complexity of 
Algorithm \ref{alg:Minimax-AIPE-inner} to implement desired oracles for $\Psi(\vy ; \bar \vx)$.
% The complexity of  is stated in the following theorem.
% [TODO]
% As we have discussed, the inexact proximal oracle of $\Phi(\vy; \bar \vx)$ can be obtained by finding the saddle point of the function $h(\vx,\vy; \bar \vx; \bar \vy):= f(\vx,\vy) + \frac{\gamma}{3} \Vert \vx - \bar \vx \Vert^3 - \frac{\gamma}{3} \Vert \vy - \bar \vy \Vert^3$. It only requires $\tilde \gO(1)$ complexity as the condition number is only a constant if setting $\gamma \lesssim \rho$. 

\begin{thm}[Inner-Loop Complexity] \label{thm:ANPE-inner}
Let $\zeta_3 = \Omega(1/ {\rm poly} (\rho, \ell, L, D,\gamma, \mu_y^{-1}, \zeta_2^{-1} ))$. Under Assumption \ref{asm:D}, \ref{asm:function-lip}, \ref{asm:grad-lip}, \ref{asm:Hess-lip}, \ref{asm:UC-UC}, and \ref{asm:M-minimax},
Algorithm \ref{alg:Minimax-AIPE-inner} returns a $(\delta,\gamma)$-proximal oracle for $\Psi(\vy ; \bar \vx)$ that satisfies Assumption \ref{asm:prox-Psi} ($\delta \lesssim \mu_y \zeta_2^4 / D^2$) in $T_{\rm saddle} (\rho+2\gamma, \gamma/2) \log (D / \zeta_3)$ iterations of $\gM_{\rm saddle}$. We can also obtain the $\delta$-zeroth-order and first-order oracles for $\Psi(\vy; \bar \vx)$ in $T_{\rm min}(\rho+ 2 \gamma, \gamma/2) \log (D/ \zeta_3)$ iterations of $\gM_{\rm min}$.
\end{thm}

\subsection{Main Result: Accelerate Existing Algorithms} \label{subsec:acc-practice}

% The final complexity of Minimax-AIPE depends on the choice of the minimization algorithm $\gM_{\rm min}$ and the minimax algorithm $\gM_{\rm saddle}$.
From the analysis in the previous section (Theorem \ref{thm:outer}, \ref{thm:middle} and \ref{thm:ANPE-inner}), the total complexity of Minimax-AIPE is proportional to 
\begin{align*}
    \left( \frac{\gamma}{\mu_x} \right)^{2/7} \left( \frac{\gamma}{\mu_y} \right)^{2/7} T_{\rm saddle} (\rho + 2\gamma,\gamma/2) + {\text{the cost involved in } T_{\rm min}},
\end{align*}
where $T_{\rm min}$ and $T_{\rm saddle}$ are the oracle complexity of the algorithm $\gM_{\rm min}$ and $\gM_{\rm saddle}$ defined in Assumption \ref{asm:M-min} and \ref{asm:M-minimax}.  As minimization problems are typically easier to solve than minimax problems, in many practical scenarios the bottleneck of the complexity depends on the first term that involved in $T_{\rm saddle}$, which is exactly the quantity that the best hyper-parameter $\gamma$ should minimize.
Below, we show the acceleration for existing methods with the optimal choice of $\gamma$.

\paragraph{Accelerating Newton Proximal Extragradient to Achieve the $\tilde \gO(\epsilon^{-4/7})$ Complexity.}
If we take $\gM_{\rm saddle} = \text{NPE-restart}$ and $\gM_{\rm min} = \text{ANPE-restart}$, then
it holds $T_{\rm saddle}(\rho, \mu)=(\rho/ \mu)^{2/3}$ 
by Theorem \ref{thm:NPE-restart} and  $T_{\rm min}(\rho, \mu) = (\rho/ \mu)^{2/7}$ by Corollary \ref{cor:ANPE}. 
% The  (NPE) method \citep{monteiro2012iteration} or its variants \citep{jiang2024adaptive,jiang2022generalized,lin2022perseus,lin2022explicit,adil2022optimal,huang2022approximation,alves2023search} achieves the upper bound of $\gO(\epsilon^{-2/3})$ for convex-concave problems. Building on these methods, it is easy to obtain NPE-restart with  by incorporating the restart scheme (see ). 
 Setting $\gamma = \rho$ and plugging in the above values of $T_{\rm min}$ and $T_{\rm saddle}$ gives an upper bound of $\tilde \gO\left( (\rho / \mu_x)^{2/7} (\rho / \mu_y)^{2/7}  \right)$ under Assumption \ref{asm:UC-UC}. 
Then by the reduction in Lemma \ref{lem:reduction-UCUC}, it indicates an upper bound of $\tilde \gO \left( (\rho D_x^3 / \epsilon)^{2/7} (\rho D_y^3 / \epsilon)^{2/7} \right)$
for finding an $\epsilon$-solution, as stated in the following theorem.
\begin{thm} \label{thm:Minimax-AIPE-UCUC}
Under Assumption \ref{asm:D}, \ref{asm:function-lip}, \ref{asm:grad-lip}, \ref{asm:Hess-lip} and \ref{asm:UC-UC}, Algorithm \ref{alg:Minimax-AIPE} with $\gamma =\rho$, $\gM_{\rm min}= \text{ANPE-restart}$, and $\gM_{\rm saddle} = \text{NPE-restart}$ finds an $\epsilon$-solution to Problem (\ref{prob:main}) second-order oracle calls bounded by $\tilde \gO\left( (\rho/\mu_x)^{2/7} (\rho/\mu_y)^{2/7} \right)$.
% \begin{align} \label{complexity-main}
%     \gO \left( \left( \frac{\gamma D_x^3}{\mu_x} \right)^{2/7} \left( \frac{\gamma D_y^3}{\mu_y} \right)^{2/7} \left( \frac{\rho +2\gamma}{\gamma} \right)^{2/3} \log \left( \frac{D}{\zeta_1} \right) \log \left( \frac{D}{\zeta_2} \right) \log \left( \frac{D}{\zeta_3} \right) \right).
% \end{align}
% In particular, setting $\gamma = \rho$ the second-order oracle complexity is 
\end{thm}

With the above analyses, our algorithm achieves the $\tilde \gO(\epsilon^{-4/7})$ complexity for second-order convex-concave minimax optimization, and our result significantly improves the existing $\gO(\epsilon^{-2/3})$ upper bound. The formal statements are presented as follows.
\begin{thm}[Main Theorem] \label{thm:Minimax-AIPE-CC}
Under Assumption \ref{asm:CC}, \ref{asm:D}, \ref{asm:function-lip}, \ref{asm:grad-lip}, \ref{asm:Hess-lip}, 
using the reduction by Lemma \ref{lem:reduction-UCUC}, Algorithm \ref{alg:Minimax-AIPE}, with $\gamma =\rho$ (defined in Hessian Lipchitzness), $\gM_{\rm min}= \text{ANPE-restart}$, and $\gM_{\rm saddle} = \text{NPE-restart}$, can solve the $\epsilon$-regularized minimax function and find an $\epsilon$ -solution to the problem (\ref{prob:main}) with second-order oracle calls bounded by $\tilde \gO\left( D_x^{6/7} D_y^{6/7} (\rho/ \epsilon)^{4/7}\right)$.
\end{thm}

\paragraph{Accelerating Lazy Extra Newton to Further Reduce the Computational Cost.}
We can further reduce the complexity by involving the idea of lazy Hessian updates \citep{doikov2023second,chen2025computationally}. 
Specifically, we take $\gM_{\rm saddle} = \text{LEN-restart}$ and  $\gM_{\rm min} = \text{ALEN-restart}$ \citep{chen2025computationally}.
% In addition to improving the oracle complexity, we show below that our framework can be easily adapted to other algorithms, such as the Lazy Extra Newton (LEN) method and its acceleration A-LEN \citep{chen2025computationally}.
Then we know from \citet{chen2025computationally} (see also Appendix \ref{apx:LEN-const} for a brief review) that $T_{\rm saddle}(\rho, \mu) = m+ m^{2/3} (\rho / \mu)^{2/3}$ and $T_{\rm min}(\rho, \mu) = m+ m^{5/7} (\rho/ \mu)^{2/7}$. Setting $\gamma = \rho / \sqrt{m}$ in Minimax-AIPE yields an upper bound of $\tilde \gO\left( m+ m^{5/7} (\rho / \mu_x)^{2/7} (\rho / \mu_y)^{2/7}  \right)$ under Assumption~\ref{asm:UC-UC}, which also indicates a corresponding 
$\tilde \gO\left( m+ m^{5/7} (\rho D_x^3/ \epsilon)^{2/7} (\rho D_y^3 / \epsilon)^{2/7}  \right)$ upper bound for convex-concave minimax problems using the reduction by Lemma \ref{lem:reduction-UCUC}.

% We summarize the theoretical guarantee of Minimax-AIPE by combining the above theorems.

% \subsection{Reducing the Computational Complexity with Lazy Hessians}

% [Chen et al.] proposed Lazy Extra Newton (LEN) and its acceleration (A-LEN) as the lazy versions of NPE and A-NPE for reducing Hessian evaluations in second-order methods. LEN and A-LEN queries new Hessian every $m$ steps, which can lead to an improvement over NPE and A-NPE that queries Hessians every steps.

% As our Minimax-AIPE builds on NPE and A-NPE. 

% % We show a $\tilde \gO(m^{3/7} \epsilon^{-4/7})$ complexity upper bounds using lazy Hessian updates. We replace NPE and ANPE with LEN and ALEN, and set $\gamma = \rho / m$ to show the result.

\begin{thm} \label{thm:Minimax-AIPE-UCUC-lazy}
Under Assumption \ref{asm:D}, \ref{asm:function-lip}, \ref{asm:grad-lip}, \ref{asm:Hess-lip} and \ref{asm:UC-UC}, Algorithm \ref{alg:Minimax-AIPE} with $\gamma =\rho/\sqrt{m}$, $\gM_{\rm min} = \text{ALEN-restart}$, $\gM_{\rm saddle} = \text{LEN-restart}$ finds an $\epsilon$-solution to Problem (\ref{prob:main}) with 
\begin{align*}
    \tilde \gO\left(m + m^{5/7} (\rho/\mu_x)^{2/7} (\rho/ \mu_y)^{2/7}  \right)
\end{align*}
steps, and the ultimate algorithm only queries Hessians every $m$ steps.
% \begin{align} \label{complexity-main}
%     \gO \left( \left( \frac{\gamma D_x^3}{\mu_x} \right)^{2/7} \left( \frac{\gamma D_y^3}{\mu_y} \right)^{2/7} \left( \frac{\rho +2\gamma}{\gamma} \right)^{2/3} \log \left( \frac{D}{\zeta_1} \right) \log \left( \frac{D}{\zeta_2} \right) \log \left( \frac{D}{\zeta_3} \right) \right).
% \end{align}
% In particular, setting $\gamma = \rho$ the second-order oracle complexity is 
\end{thm}
Then by applying similar arguments, we get the theorem below where we achieve accelerated rates than the original LEN method \citep{chen2025computationally}.
\begin{thm} \label{thm:Minimax-AIPE-CC-lazy}
Under Assumption \ref{asm:CC}, \ref{asm:D}, \ref{asm:function-lip}, \ref{asm:grad-lip}, \ref{asm:Hess-lip}, Algorithm \ref{alg:Minimax-AIPE} with $\gamma =\rho/\sqrt{m}$, $\gM_{\rm min} = \text{ALEN-restart}$, $\gM_{\rm saddle} = \text{LEN-restart}$ finds an $\epsilon$-solution to Problem (\ref{prob:main}) with 
\begin{align*}
    \tilde \gO\left(m + m^{5/7} (D_x^3 \rho /\epsilon)^{2/7} (D_y^3 \rho / \epsilon)^{2/7}  \right)
\end{align*}
steps, and the ultimate algorithm only queries Hessians every $m$ steps.
% Using the reduction by Lemma \ref{lem:reduction-UCUC}, running the lazy version of Algorithm \ref{alg:Minimax-AIPE} with $\gamma = \rho / \sqrt{m}$ on an $\epsilon$-regularized minimax function can provably find an $\epsilon$-solution to Problem (\ref{prob:main}) with $\tilde \gO\left(m + m^{5/7} D_x^{6/7} D_y^{6/7} (\rho/ \epsilon)^{4/7}  \right)$ steps.
\end{thm}

\section{Conclusion and Future Works}

In this paper, we propose the Minimax-AIPE algorithm for convex-concave minimax problems. Our theoretical result shows our proposed method can achieve the convergence rate of $\tilde \gO(\epsilon^{-4/7})$, which refutes the common conjecture that the optimal rate for this setting is $\Theta(\epsilon^{-2/3})$. Our framework is also compatible with lazy Newton methods, and it yields an upper bound of $\tilde \gO(m +m^{5/7} \epsilon^{-4/7})$ by reusing Hessians every $m$ iterations. We discuss some potential future directions as follows.

\paragraph{Shaving Logarithmic Factors in our Upper Bound} 
Our upper bound includes logarithmic factors is $\gO( \epsilon^{-4/7} \log^3(\epsilon^{-1}))$. The logarithmic factors may be possibly removed using techniques to achieve optimal first-order oracle complexity in strongly-convex-strongly-concave problems \citep{kovalev2022firstSC}. 

\paragraph{More Acceleration Results}
Our framework can be applied to accelerate any existing linear convergent algorithms for uniformly-convex-uniformly-concave functions (Assumption \ref{asm:M-minimax}).  Potentially, it can be used to accelerate quasi-Newton methods 
\citep{jiang2023online}. But it requires additional analysis to adapt the result of \citep{jiang2023online} from strongly-convex-strongly-concave functions to uniformly-convex-uniformly-concave cases. 
It is also possible to use our framework to accelerate stochastic algorithms  \citep{chayti2023unified,lin2022explicit,wang2019stochastic,zhou2019stochastic,tripuraneni2018stochastic} like in the original Catalyst method \citep{lin2018catalyst}.

\paragraph{Lower Bounds} It would also be important for future works to establish lower bounds for this problem. The key is to carefully design a zero-chain that couples the worse-case instances for both variables $\vx$ and $\vy$. Although tight lower bounds have been proved for first-order minimax optimization \citep{ouyang2021lower,xie2020lower,zhang2022lower,li2021complexity}, this problem remains open for second-order methods.

\section*{Acknowledgment}
Jingzhao Zhang is supported by National Key R\&D Program of China 2024YFA1015800 and Shanghai Qi Zhi Institute Innovation Program. 
Luo Luo is supported by the Major Key Project of PCL under Grant PCL2024A06, the National Natural Science Foundation of China (No. 62206058), and Shanghai Basic Research Program (23JC1401000).
Chengchang Liu is supported by the National Natural Science Foundation of China (624B2125). 

\bibliography{sample}

\newpage
\appendix

\section{Discussions on the Existing $\Omega(\epsilon^{-2/3})$ Lower Bound} \label{apx:diss-lower}

% \citet{adil2022optimal} proved a lower bound by assuming a second-order method is applied on the primal function $\Phi(x) = \max_{\vy \in \gY} f(\vx,\vy)$. But we should note that in proving upper bounds the algorithms are able to access information of $f(\vx,\vy)$ instead of $\Phi(\vx)$. 
\citet{adil2022optimal} provided an $\Omega(\epsilon^{-2/3})$ lower bound on the second-order oracle of the primal function $\Phi(x) = \max_{\vy \in \gY} f(\vx,\vy)$. 
However, practical algorithms (including our algorithm and \citet{adil2022optimal}'s algorithm)
%and analysis upper bound 
use information of $f(\vx,\vy)$ instead of $\Phi(\vx)$. 
It is clear that the oracle of $\Phi(\vx)$ is quite different from the oracle of $f(\vx,\vy)$.

%$( f(\vx,\vy), \nabla f(\vx,\vy), \nabla^2 f(\vx,\vy))$.
\citet{lin2022perseus} adapted the hard instance of \citep{adil2022optimal} to establish a lower bound with the second-order oracle of $f(\vx,\vy)$.
They showed that there exists a $\rho$-Hessian smooth function, such that the output of any second-order method that queries no more than $T$ second-order oracles of $f(\vx,\vy)$ must has duality gap of $\Omega(\rho D_x D_y^2 / T^{3/2})$ \citep[Theorem 3.11]{lin2022perseus}.
However, the diameters of sets $\gX$ and $\gY$ in their construction do depend on $T$. Specifically, their hard instance requires taking $D_x = \Theta(T^{3/2})$ and $D_y = \Theta(T^{1/2})$.
Hence, their theorem cannot directly give a lower bound $\Omega(\epsilon^{-2/3})$ when $D_x$ and $D_y$ are constants. 
Furthermore, we find
that, unlike the case of first-order methods \citep{zhang2022lower}, scaling down the diameters $D_x$ and $D_y$ in the hard instance of minimax problems is not easy, because the analysis heavily relies on their specific $\gX$ and $\gY$. 
Therefore, the existing lower bounds do not apply to our algorithm.

\section{Proofs in Section \ref{sec:AIPE}}

\subsection{Proof of Lemma \ref{lem:CRN-is-MS}}

% let $\vz = (\vx,\vy)$ and $\mF(\vz) :=
% \begin{bmatrix}
% \nabla_x f(\vx,\vy) \\
% - \nabla_y f(\vx, \vy) 
%  \end{bmatrix}
% $. The CRN solution (Definition \ref{dfn:cubic-VI}) to minimax problem at point $\bar \vz = ( \bar \vx, \bar \vy)$ is $(\vx,\vy) = (\bar \vx+ \Delta \vx, \bar \vy + \Delta \vy) $, where
% \begin{align*}
%      (\Delta \vx, \Delta \vy) = \arg \min_{\Delta \vx \in \gX}\max_{\Delta \vy \in \gY} \left\{ 
%      \langle \nabla f (\bar \vz), \Delta \vz \rangle + \frac{1}{2} \langle  \nabla^2 f(\bar \vz) \Delta \vz, \Delta \vz \rangle + \frac{\rho}{3} \Vert \Delta \vx \Vert^3 - \frac{\rho}{3} \Vert \Delta \vy \Vert^3
%      \right\}.
% \end{align*}

\begin{proof}
We prove a general result that includes both minimax and minimization problems.
    From the first-order optimality condition of the CRN oracle (Definition \ref{dfn:cubic-VI}), we know that $\vz = {\rm CRN}(\bar \vz, 2\rho)$ satisfies
\begin{align} \label{eq:close-CRN}
    \vu :=  -\left(\mF(\bar\vz) + \nabla \mF(\bar \vz) (\vz - \bar \vz) +  \lambda
    (\vz- \bar \vz) \right) \in 
    \begin{bmatrix}
    \partial \gI_{\gX}(\vx) \\
    -\partial \gI_{\gY}(\vy)
\end{bmatrix},
\end{align}
where $\lambda = \rho \Vert \vz- \bar \vz \Vert$.
Then using Assumption \ref{asm:Hess-lip}, we have that
\begin{align} \label{eq:CRN-MS}
    \Vert \mF(\vz) + \vu + \lambda (\vz - \bar \vz) \Vert &\le \Vert \mF(\vz) - \mF(\bar \vz) - \nabla \mF(\bar \vz) (\vz - \bar \vz) \Vert \le \frac{\rho}{2} \Vert  \vz- \bar \vz \Vert^2,
\end{align}
which satisfies a $(0, \rho )$-proximal oracle according to Definition \ref{dfn:inexact-MS-oracle}.
\end{proof}

\subsection{Proof of Theorem \ref{thm:ANPE-restart}}

\begin{proof}
By Theorem \ref{thm:ANPE}, each epoch of Algorithm \ref{alg:ANPE-restart} ensures $\Vert \vz^{(s+1)}- \vz^* \Vert \le \frac{1}{2} \Vert \vz^{(s)}- \vz^* \Vert$ if setting $T = \gO\left((\gamma/ \mu)^{2/7}\right)$. And therefore 
the algorithm finds a point $\vz^{(S)}$ such that $\Vert \vz^{(S)} - \vz^* \Vert \le \epsilon$ in $S = \left \lceil \log_2 (d_0 /\epsilon) \right \rceil$ epochs.
%     The goal of each epoch in Algorithm \ref{alg:ANPE-restart} is to ensure that
% \begin{align*}
%     \frac{\mu}{3} \Vert \vz^{(s+1)} - \vz^* \Vert^3 \le \frac{\mu}{6} \Vert \vz^{(s)} - \vz^* \Vert^3.
% \end{align*}
% By Lemma \ref{lem:UC-grad-dominant}, it suffices to find 
% \begin{align*}
%     h(\vz^{(s)}) - h(\vz^*) \le \frac{\mu}{6} \Vert \vz^{(s)} - \vz^* \Vert^3 := \epsilon^{(s)}.
% \end{align*}
% Therefore, we can invoke Theorem \ref{thm:ANPE} with $\epsilon = \epsilon^{(s)} \ge \frac{\mu \epsilon^3}{6}$ in each epoch to get the desired result. 
\end{proof}

% Now we prove Theorem \ref{thm:ANPE}.

\begin{thm} \label{thm:ANPE}
Under the same setting as Theorem \ref{thm:ANPE-restart}, running Algorithm \ref{alg:ANPE} finds a point $\vz_t$ such that $\Vert\vz_t - \vz^*\Vert \le 2 \epsilon$ in $T= \gO\left( \left ({\gamma d_0^3}/{(\mu \epsilon^3)} \right)^{2/7}  \right)$ iterations if $\delta \le {\mu \epsilon^4}/{(144 D^2)}$, where $d_0 = \Vert \vz_0 - \vz^*\Vert$.
\end{thm}

% \begin{remark}
% In fact, the above theorem only guarantees there exists an iterate $\vz \in \{\vz_t,\tilde \vz_t \}$ such that $h(\vz) - h(\vz^*) \le \frac{\mu \epsilon^3}{3}$. Therefore we can restrict our analysis in the iterates $t \le T_{\epsilon}$, where $T_{\epsilon} = \arg \min_{t \ge 0} \{ \Vert \vz - \vz^* \Vert \le \mu \epsilon^3 / 3, ~ \vz \in \{\vz_t, \tilde \vz_t \} \}$. This restriction allows us to use a bounded precision $\delta$, and the similar technique 
% also appears in \citep{bubeck2019complexity}.
% A small difference to \citep{carmon2022optimal} is that now the algorithm should not simply output $\vz_T$, as when $t > T_{\epsilon}$ the estimated $\delta$ may not be sufficient to ensure the descent of $h(\vz)$. Instead, we should output $\vz_{\rm out} = \arg \min_{0 \le t \le T} \{ h(\vz), \vz \in \{\vz_t,\tilde \vz_t \} \}$ when
% $h(\vz)$ is exactly computable. In our Minimax-AIPE (Algorithm \ref{alg:Minimax-AIPE}), $h(\vz)$ is $P(\vx)= \max_{\vy \in \gY} f(\vx,\vy)$ or $D(\vy; \bar\vx) = \min_{\vx \in \gX} g(\vx,\vy; \bar \vx)$ and we can easily find $\tilde h(\vz)$ such that $\vert \tilde h(\vz) - h(\vz) \vert \le \mu \epsilon^3/3$ by solving the maximization (or minimization) via A-NPE. Then we can output $\vz_{\rm out} = \arg \min_{0 \le t \le T} \{ \tilde h(\vz), \vz \in \{\vz_t, \tilde \vz_t \} \}$, which guarantees $\Vert \vz_{\rm out} - \vz^* \Vert \le 2 \epsilon$.
% \end{remark}

\begin{proof}
Some of our proof of error tolerance properties is motivated by the proof of \citep[Theorem 7]{bubeck2019complexity}. Compare to \citep{bubeck2019complexity}, our analysis is more simple as we do not need to analyze the additional line search with inexact proximal oracles as \citep[Section E]{bubeck2019complexity}. Moreover, our analysis tackles the constrained case, but \citep{bubeck2019complexity} only considers the unconstrained case, \textit{i.e.} $\gZ = \sR^d$. 
We define $E_t:= h(\vz_t) - h(\vz^*)$, $R_t:= \frac{1}{2} \Vert \vv_t - \vz^* \Vert^2$ and $N_{t+1}:= \frac{1}{2} \Vert \tilde \vz_{t+1} - \bar \vz_t \Vert^2$. 
From Lemma \ref{lem:UC-grad-dominant}, to find a point $\Vert \vz_t - \vz^* \Vert \le \epsilon$ it suffices to find $E_t \le {\mu \epsilon^3}/{3}$.
Let ${\rm Proj}_{\gZ}(\bar \vz ) = \arg \min_{\vz \in \gZ} \Vert \vz - \bar \vz \Vert$ be the projection operator of $\vz$ onto the set $\gZ$. The update rule of $\vv_t$ gives
\begin{align} \label{eq:MS-1}
\begin{split}
    R_{t+1} &= \frac{1}{2} \left \Vert {\rm Proj}_{\gZ} (\vv_t - a_{t+1}  (\vg_{t+1} + \vu_{t+1})) - \vz^* \right \Vert^2 \\
    &\le \frac{1}{2} \left \Vert (\vv_t - a_{t+1} (\vg_{t+1} + \vu_{t+1})) - \vz^* \right \Vert^2 \\
    &= R_t + a_{t+1} \langle \vg_{t+1} + \vu_{t+1}, \vz^* - \vv_t \rangle + \frac{a_{t+1}^2}{2} \Vert \vg_{t+1} + \vu_{t+1} \Vert^2 \\
    &\le  R_t + a_{t+1} \langle  \nabla h(\tilde \vz_{t+1}) + u_{t+1}, \vz^* - \vv_t \rangle + a_{t+1}^2 \Vert \nabla h(\tilde \vz_{t+1}) + \vu_{t+1} \Vert^2 \\
    &\quad + a_{t+1} \delta \Vert \vv_t - \vz^* \Vert + a_{t+1}^2 \delta^2.
    % &\le D_t + a_{t+1} \langle \nabla f(\tilde \vz_{t+1}), \vz^* - \vv_t \rangle + a_{t+1}^2 \Vert \nabla f(\tilde \vz_{t+1}) \Vert^2 + a_{t+1} D \delta_{t+1} + a_{t+1}^2 \delta_{t+1}^2.
\end{split}
\end{align}
By the update rule of $\bar \vz_t = \frac{A_t}{A_{t+1}'} \vz_t + \frac{a_{t+1}'}{A_{t+1}'} \vv_t $ and $A_{t+1}' = A_t + a_{t+1}'$, we  have
\begin{align*}
    a_{t+1}' \vv_t= A_{t+1}' \bar \vz_t  - A_t \vz_t = a_{t+1}' \tilde \vz_{t+1} + A_t (\vz_t - \tilde \vz_{t+1}) + A_{t+1}' (\bar \vz_t - \tilde \vz_{t+1}) .
\end{align*}
Then subtracting $a_{t+1}' \vz^*$ and taking inner product with $ \nabla h(\tilde \vz_{t+1}) + \vu_{t+1} $ yields
\begin{align}\label{eq:to-plug}
\begin{split} 
    &\quad a_{t+1}' \langle \nabla h(\tilde \vz_{t+1}) + \vu_{t+1}, \vz^* - \vv_t \rangle \\
&= \langle \nabla h(\tilde \vz_{t+1}) + \vu_{t+1}, a_{t+1}'  ( \vz^*-  \tilde \vz_{t+1}) + A_t (\vz_t - \tilde \vz_{t+1})  + A_{t+1}' (\tilde \vz_{t+1}- \bar \vz_t) \rangle  \\
&\le a_{t+1}' ( h(\vz^*) - h(\tilde \vz_{t+1})  )  + A_t (h(\vz_t) - h(\tilde \vz_{t+1})) + A_{t+1}' \langle \nabla h(\tilde \vz_{t+1}) + \vu_{t+1}, \tilde \vz_{t+1} - \bar \vz_t \rangle,
\end{split}
\end{align}
where in the last step we use the convexity of $h$ and the fact that $\vu_{t+1} \in \partial \gI_{\gZ}(\tilde \vz_{t+1})$.
We continue to upper bound the inner product term with 
\begin{align*}
     &\quad \langle \nabla h(\tilde \vz_{t+1}) + \vu_{t+1}, \tilde \vz_{t+1} - \bar \vz_t \rangle \\
     &=\frac{1}{2 \lambda_{t+1}} \Vert \nabla h(\tilde \vz_{t+1}) + \vu_{t+1} + \lambda_{t+1} (\tilde \vz_{t+1} - \bar \vz_t) \Vert^2 \\
     &\quad - \frac{1}{2 \lambda_{t+1} } \Vert \nabla h(\tilde \vz_{t+1}) + \vu_{t+1} \Vert^2  - \frac{\lambda_{t+1}}{2} \Vert \tilde \vz_{t+1} - \bar \vz_t \Vert^2.
\end{align*}
Substituting this upper bound into Eq. (\ref{eq:to-plug}) yields
\begin{align*}
    &\quad a_{t+1}' \langle \nabla h(\tilde \vz_{t+1}) + \vu_{t+1}, \vz^* - \vv_t \rangle \\
    &\le a_{t+1}' ( h(\vz^*) - h(\tilde \vz_{t+1})  )  + A_t (h(\vz_t) - h(\tilde \vz_{t+1})) - \frac{A_{t+1}'}{2 \lambda_{t+1} } \Vert \nabla h(\tilde \vz_{t+1}) + \vu_{t+1} \Vert^2 \\
    &\quad + \frac{A_{t+1}'}{2 \lambda_{t+1}} \Vert \nabla h(\tilde \vz_{t+1}) + \vu_{t+1} + \lambda_{t+1} (\tilde \vz_{t+1} - \bar \vz_t) \Vert^2  - \frac{A_{t+1}' \lambda_{t+1}}{2} \Vert \tilde \vz_{t+1} - \bar \vz_t \Vert^2.
\end{align*}
Furthermore, 
by Definition \ref{dfn:inexact-MS-oracle}, we have that 
\begin{align} \label{eq:MS-2}
\begin{split}
     &\quad a_{t+1}' \langle \nabla h(\tilde \vz_{t+1}) + \vu_{t+1}, \vz^* - \vv_t \rangle \\
    &\le a_{t+1}' ( h(\vz^*) - f(\tilde \vz_{t+1})  )  + A_t (h(\vz_t) - h(\tilde \vz_{t+1})) \\
    &\quad - \frac{A_{t+1}'}{2 \lambda_{t+1} } \Vert \nabla h(\tilde \vz_{t+1}) + \vu_{t+1} \Vert^2 - \frac{3}{4} \lambda_{t+1} A_{t+1}'  N_{t+1} + \frac{A_{t+1}'}{2 \lambda_{t+1}} \delta^2.
\end{split}
\end{align}
% It is easy to check that our inexact second-order proximal oracle (Definition \ref{dfn:iprox}) satisfies the assumption of MS oracle in \citep[Definition 1]{carmon2022optimal} with $\sigma = 1/2$. Define $f_{\gamma}(\vz; \bar \vz) =  f(\vz) + \frac{\gamma}{3} \Vert \vz- \bar \vz \Vert^3$. By Definition \ref{dfn:iprox}, we have
% \begin{align*}
%     \Vert \nabla f_{\gamma}(\tilde \vz_{t+1}; \bar \vz_t) +\vu_{t+1} \Vert  &\le \frac{\gamma}{6} \Vert \tilde \vz_{t+1} - \hat \vz_t \Vert^2 \\
%     &\le \frac{\gamma}{3} \Vert \tilde \vz_{t+1} - \hat \vz_t \Vert^2 + \frac{\gamma}{3} \Vert \bar \vz_t- \hat \vz_t \Vert^2 \\
%     &\le  \frac{\gamma}{3} \Vert \tilde \vz_{t+1} - \hat \vz_t \Vert^2 + \frac{2}{3} \Vert \nabla f_{\gamma}(\tilde \vz_{t+1}; \bar \vz_t) +  \vu_{t+1} \Vert,
% \end{align*}
% where $\hat \vz_t = \arg \min_{\vz \in \gZ}$
% Therefore, we can use the same steps as \citep[Theorem 1]{carmon2022optimal} to obtain
% \begin{align} \label{2}
% \begin{split}
%      &\quad A_{t+1}' (f(\tilde \vz_{t+1}) - f(\vz^*) \\
%     &\le A_t E_t + a_{t+1}' \langle \nabla f(\tilde \vz_{t+1}, \vv_t - \vz^* \rangle -\frac{3}{4} A_{t+1}' \lambda_{t+1} N_{t+1} - \frac{A_{t+1}'}{2 \lambda_{t+1}} \Vert \nabla f(\tilde \vz_{t+1}) \Vert^2,
% \end{split}
% \end{align}
Now we separately consider the two cases $\lambda_{t+1} \le \lambda_{t+1}'$ (\textit{Case I}) and $\lambda_{t+1} > \lambda_{t+1}'$ (\textit{Case II}). In the first case, we have that $\vz_{t+1} = \tilde \vz_{t+1} $, $a_{t+1} = a_{t+1}'$, $A_{t+1} = A_{t+1}'$ and $A_{t+1} = 2 \lambda_{t+1}' a_{t+1}^2$. Then we can directly sum up Eq. (\ref{eq:MS-1}) and Eq. (\ref{eq:MS-2}) to obtain that
\begin{align*}
\textit{Case I:} & \quad A_{t+1} E_{t+1} + D_{t+1} + \frac{3}{4} \lambda_{t+1} A_{t+1} N_{t+1} \\
&\le A_t E_t + D_t + a_{t+1} \delta \Vert \vv_t - \vz^* \Vert + 2 a_{t+1}^2 \delta^2.
\end{align*}
In the second case, we have that $A_{t+1} = (1-\gamma_{t+1}) A_t + \gamma_{t+1} A_{t+1}'$ with $\gamma_{t+1}= \lambda_{t+1}' / \lambda_{t+1}$. By the convexity of $h$, we  have that
\begin{align*}
    A_{t+1} E_{t+1} &\le (1- \gamma_{t+1}) A_t E_t + \gamma_{t+1} A_{t+1}' (h(\tilde \vz_{t+1} ) - h(\vz^*)).
\end{align*}
We use Eq. (\ref{eq:MS-2}) multiplied by $\gamma_{t+1}$ to upper bounding the above inequality as
\begin{align} \label{eq:MS-3}
\begin{split}
       A_{t+1} E_{t+1} &\le A_t E_t + a_{t+1} \langle \nabla h(\tilde \vz_{t+1}) + \vu_{t+1}, \vz^* - \vv_t \rangle \\
    &- \frac{3}{4} \lambda_{t+1}' A_{t+1}' N_{t+1} - \frac{\gamma_{t+1} A_{t+1}'}{2 \lambda_{t+1}} \Vert \nabla h(\tilde \vz_{t+1}) + \vu_{t+1} \Vert^2 + \frac{\gamma_{t+1} A_{t+1}'}{2 \lambda_{t+1}} \delta^2, 
\end{split}
\end{align}
where we also use facts $ \gamma_{t+1} a_{t+1}' = a_{t+1}$ and $\lambda_{t+1}' = \gamma_{t+1} \lambda_{t+1}$. Note that
\begin{align*}
    \frac{\gamma_{t+1} A_{t+1}'}{2 \lambda_{t+1}} = \frac{\gamma_{t+1}^2 A_{t+1}' }{\lambda_{t+1}'} = (\gamma_{t+1} a_{t+1}') = a_{t+1}^2.
\end{align*}
Summing up Eq. (\ref{eq:MS-1}) and Eq. (\ref{eq:MS-3}) yields that
\begin{align*}
    \textit{Case II:} &\quad A_{t+1} E_{t+1} + D_{t+1} + \frac{3}{4} \lambda_{t+1}' A_{t+1}' N_{t+1} \\
    &\le A_t E_t + D_t + a_{t+1} \delta \Vert \vv_t - \vz^* \Vert + 2 a_{t+1}^2 \delta^2.
\end{align*}
The inequality for both Case I and Case II can be unified as
\begin{align} \label{eq:MS-unify}
    A_{t+1} E_{t+1} +R_{t+1} + \frac{3}{4} \min\{ \lambda_{t+1}, \lambda_{t+1}'\} A_{t+1}' N_{t+1} \le A_t E_t + R_t + \delta_{t+1},
    \end{align}
where $\delta_{t+1}:=  a_{t+1} \delta \Vert \vv_t - \vz^* \Vert + 2 a_{t+1}^2 \delta^2$.
Let $T_{\epsilon}$ be the first time that the algorithm achieves a point $\vz \in\{\vz_t, \tilde \vz_{t} \}$  such that $h(\vz) - h(\vz^*) \le {\mu \epsilon^3}/{3} $, \textit{i.e.} $T_{\epsilon} = \arg \min_{t \ge 0} \{ h(\vz) - h(\vz^*) \le \mu \epsilon^3/3, ~ \vz \in \{\vz_t, \tilde \vz_t \} \}$. 
We let $A_{T_{\epsilon}} \delta \le c \Vert \vz_0 - \vz^* \Vert$ for some constant $c>0$, and we will specify the value of $c$ later. From Eq. (\ref{eq:MS-unify}) and $A_t = \sum_{j=0}^t a_t$ we know that
\begin{align*}
    \frac{1}{2} \Vert \vv_t - \vz^* \Vert^2 \le \frac{1}{2} \Vert \vz_0 - \vz^* \Vert^2 + c \Vert \vz_0 - \vz^* \Vert \max_{0 \le j \le t} \Vert \vv_j - \vz^* \Vert + 2 c^2 \Vert \vz_0 - \vz^* \Vert^2.
\end{align*}
Since the above inequality holds for any $0 \le t \le T$, it must hold for $t_m \in \arg \max_{0 \le j \le T} \Vert \vv_j - \vz^* \Vert$. Therefore, we have that
\begin{align*}
    \Vert v_{t_m} - \vz^* \Vert^2 \le (1+  4 c^2 )\Vert \vz_0 - \vz^* \Vert^2 + 2 c \Vert \vz_0 - \vz^* \Vert \Vert \vv_{t_m} - \vz^* \Vert.
\end{align*}
Solving the above quadratic with respect to variable $\Vert v_{t_m} - \vz^* \Vert$ yields
\begin{align*}
\Vert v_{t_m} - \vz^* \Vert \le \left(2c+ \sqrt{1 + 5c^2} \right) \Vert \vz_0 - \vz^* \Vert \le (1+5c) \Vert \vz_0 - \vz^* \Vert.
\end{align*}
It implies 
\begin{align} \label{eq:ub-sum-delta}
    \sum_{j=0}^t \delta_{t+1} &\le A_{t+1} \delta \Vert v_{t_m} - \vz^* \Vert + 2 A_{t+1}^2 \delta^2 \le c (1+ 7c) \Vert \vz_0 - \vz^* \Vert^2
\end{align}
Note that $A_0 = 0$.
We telescope Eq. (\ref{eq:MS-unify}) over $t=0,1,\cdots,T_{\epsilon}-1$ and substitute Eq. (\ref{eq:ub-sum-delta}) to get
\begin{align*}
    A_{T_{\epsilon}} E_{T_{\epsilon}}+ R_{T_\epsilon} + \sum_{t=0}^{T_{\epsilon}-1} \min\{ \lambda_{t+1}, \lambda_{t+1}'\} A_{t+1}' N_{t+1} \le R_0 +  c (1+ 7c) \Vert \vz_0 - \vz^* \Vert^2.
\end{align*}
We let $c \le 1/8 $ to have 
\begin{align} \label{eq:our-MS}
     A_{T_\epsilon} E_{T_\epsilon}+ R_{T_\epsilon} + \sum_{t=0}^{T_\epsilon-1} \min\{ \lambda_{t+1}, \lambda_{t+1}'\} A_{t+1}' N_{t+1} \le 2R_0.
\end{align}
Note that Eq. (\ref{eq:our-MS}) is exactly 
\citep[ Eq. 9 ]{carmon2022optimal} with their $R_0$ now being replaced by $2R_0$. We can then follow the remaining steps as the proof of \citep[Theorem 1]{carmon2022optimal}
to show that the algorithm finds $E_t \le {\mu \epsilon^3}/{3}$ in $T_{\epsilon}= \gO\left( \left( {\gamma d_0^3}/{(\mu \epsilon^3)} \right)^{2/7}  \right)$ iterations, where $d_0 = \Vert \vz_0 - \vz^* \Vert$.
Therefore, we have that $\min \{\vz_{T_{\epsilon}}, \bar \vz_{T_{\epsilon}} \}\} \le {\mu \epsilon^3}/{3}$.
This further implies
\begin{align*}
h(\vz^{\rm out}) \le \tilde h(\vz^{\rm out}) + \delta \le \tilde h(\vz_{T_{\epsilon}}) + \delta \le h(\vz_{T_{\epsilon}}) + 2 \delta. 
\end{align*}
If $2 \delta \le {\mu \epsilon^3}/{3}$ and we know from $E_t \le {\mu \epsilon^3}/{3}$ that $h(\vz_{\rm out}) - h(\vz^*) \le {2 \mu \epsilon^3}/{3} $. By Lemma \ref{lem:UC-grad-dominant}, this implies that $\Vert \vz^{\rm out} - \vz^* \Vert \le 2 \epsilon$.
The last thing to show how to fulfill the goal of $\delta \le {d_0}/{(8 A_{T_{\epsilon}})}$ to ensure the above convergence rate. 
This can be done by giving a uniform lower bound of $d_0$ and a uniform upper bound of $A_{T_\epsilon}$, as we specify below. As for $d_0$, we use the trivial bound  $d_0 \ge \epsilon$. As for $A_{T_{\epsilon}}$, we analyze its upper bound below.

From Eq. (\ref{eq:our-MS}) we know that for all $t < T_\epsilon$ we have that ${\mu \epsilon^3}/{3} < E_{t} \le 2 R_0/ A_t$, which means that $A_t < {6 R_0}/{(\mu \epsilon^3)}$ for all $t < T_{\epsilon}$. The update of $A_{t+1}' = 2 \lambda_{t+1}' \left(a_{t+1}'\right)^2$, in conjunction with $R_0 \le D^2/2$, means that
\begin{align} \label{eq:upper-at}
    a_{t}' = \frac{1 + \sqrt{1 + 8 \lambda_{t}' A_{t-1}}}{4 \lambda_{t+1}'} \le \frac{1}{2 \lambda_{t}'} + \frac{1}{2} \sqrt{\frac{2A_{t-1}}{\lambda_{t}'}} < \frac{1}{2 \lambda_{t}'} + \frac{D}{2} \sqrt{\frac{6}{\mu \epsilon^3 \lambda_{t}'}}.
\end{align}
Therefore, we need a lower bound of $\lambda_{t}'$. As an intermediate goal, we first analyze $\lambda_t$. According to Definition \ref{dfn:inexact-MS-oracle} we have that
\begin{align*}
    \frac{3 \lambda_{t}}{2} \Vert \tilde \vz_{t} - \bar \vz_{t-1} \Vert \ge  \Vert \nabla h(\tilde \vz_{t}) + \vu_{t} \Vert  - \delta \ge \frac{h(\tilde \vz_{t}) - h(\vz^*)}{\Vert \tilde \vz_{t} - \vz^* \Vert} - \delta.
\end{align*}
By the definition of $T_{\epsilon}$, we know that for all $t < T_{\epsilon}$, when $\delta \le {\mu \epsilon^3}/{(6 D)}$ we have
\begin{align} \label{eq:lambda-t-prime}
    \lambda_{t} \ge \frac{2}{3 \Vert \tilde \vz_{t} - \bar \vz_{t-1} \Vert} \left( 
    \frac{h(\tilde \vz_{t}) - h(\vz^*)}{\Vert \tilde \vz_{t} - \vz^* \Vert} - \delta
    \right) \ge \frac{\mu \epsilon^3}{9 D^2}.
\end{align}
Let $k$ the the last iterate before $t$ such that $ \lambda_k' > \lambda_k$, \textit{i.e.}, $k = \arg \max_{0 \le k' \le t} \{ \lambda_{k'}' > \lambda_{k'}\}$.
If $t=k$, then Eq. (\ref{eq:lambda-t-prime}) already gives a lower bound of $\lambda_t'$ such that $\lambda_t' > \lambda_t \ge \mu \epsilon^3 / (9D^2)$. Otherwise, if $t <k$, we know from the update rule of $\lambda_t'$ that 
\begin{align*}
    \lambda_t' =  2^{t-k-1} \lambda_{k+1}' = 2^{t-k-2} \lambda_k' > \lambda_k /2 \ge \frac{\mu \epsilon^3}{18 D^2}.
\end{align*}
Now, if $\lambda'_{T_{\epsilon}-1} > \lambda_{T_{\epsilon}-1} $, then 
\begin{align*}
    \lambda_{T_{\epsilon}}' = \frac{1}{2} \lambda'_{T_{\epsilon}-1}  > \frac{\mu \epsilon^3}{18D^2}.
\end{align*}
Else, if $\lambda'_{T_{\epsilon}-1} \le \lambda_{T_{\epsilon}-1} $, then 
\begin{align*}
     \lambda_{T_{\epsilon}}' = 2 \lambda'_{T_{\epsilon}-1} > \frac{\mu \epsilon^3}{9 D^2}.
\end{align*}
For both cases, plugging  the above inequalities into Eq. (\ref{eq:upper-at}) yields that
\begin{align*}
    a_{T_{\epsilon}}' < \frac{(9 + 3 \sqrt{3}) D^2}{\mu \epsilon^3}.
\end{align*}
Therefore, 
\begin{align*}
    A_{T_{\epsilon}} = A_{T_{\epsilon} -1} + a_{T_{\epsilon}} \le A_{T_{\epsilon} -1} + a_{T_{\epsilon}}' \le \frac{(12 + 3 \sqrt{3}) D^2}{\mu \epsilon^3}. 
\end{align*}
Hence, we can ensure $\delta \le {d_0}/{(8 A_{T_{\epsilon}})}$ by setting $\delta \le \mu \epsilon^4 / {(144  D^2)}$.
% \begin{align} \label{eq:lambda-t-prime}
%  > 2^{t-k-2} \lambda_{k} 
% \left( \frac{1}{2} \right)^{t-k} \lambda_k \ge \left( \frac{1}{2} \right)^{t-k} \frac{\mu \epsilon^3}{9 D^2}.
% \end{align}
\end{proof}

% If the algorithm has not finds a point $\vz_t$ such that $h(\vz_t) - h(\vz^*) \le \epsilon $, then we must have 

% Therefore, we can simplify set 
% \begin{align*}
%     \delta \le \frac{1}{8 D \epsilon} \le  \frac{\Vert \vz_0 - \vz^* \Vert}{8 A_T}  
% \end{align*}
% to ensure that we can get (\ref{eq:our-MS}).

% \begin{lem} \label{lem:ub-lb-lambda}
% Under the same setting as Theorem \ref{thm:ANPE-restart}, Algorithm \ref{alg:ANPE} ensures that 
% \begin{align} \label{eq:ub-lb-lambda}
%    \frac{\mu \epsilon^3}{18 D^2} \le \lambda_{t+1}' \le 2  \gamma D
% \end{align}
% for all $t < T_{\epsilon}$, where $T_{\epsilon} = \arg \min_{t \ge 0} \{ \Vert \vz - \vz^* \Vert \le \mu \epsilon^3/ 3, ~ \vz \in \{\vz_t, \tilde \vz_t \} \}$,$D = \sup_{\vz, \vz' \in \gZ}\Vert \vz- \vz' \Vert$.
% \end{lem}

% \begin{proof}
% For all $t$, we have $\lambda_{t+1} = \gamma \Vert \tilde \vz_{t+1} - \bar \vz_t \Vert \le \gamma D$. For all $ t< T_{\epsilon}$, we have 
% \begin{align*}
%     \Vert \nabla h(\tilde \vz_{t+1}) \Vert \ge \frac{h(\tilde \vz_{t+1}) - h(\vz^*)}{\Vert \tilde \vz_{t+1} - \vz^* \Vert} \ge  \frac{\mu \epsilon^3}{3 D},
% \end{align*}
% where the first inequality uses convexity of $h$ and the Cauchy-Schwarz inequality.  Let $\delta \le \frac{\mu\epsilon^3}{6 D}$, then 
% which implies $\lambda_{t+1} \ge \frac{\mu \epsilon^3}{18 D^2}$ for all $t < T_{\epsilon}$. Finally the update rule of $\lambda_{t+1}'$ indicates Eq. (\ref{eq:ub-lb-lambda}).
% \end{proof}

\section{Proofs in Section \ref{subsec:reduction}} \label{apx:proof-pre}

\begin{lem} \label{lem:U-monotone}
Let $f: \gX \times \gY \rightarrow \sR$ be a $\mu$-uniformly-convex-$\mu$-uniformly concave function. We have
\begin{align} \label{eq:U-monotone}
    \langle \vg_1 - \vg_2, \vz_1- \vz_2 \rangle \ge \frac{2\mu}{3} \Vert \vz_1 - \vz_2 \Vert^3,
\end{align}
for any $\vg_1 \in \mA(\vz_1)$, $\vg_2 \in \mA(\vz_2)$, where
\begin{align*}
    \mA(\vz) = 
    \begin{bmatrix}
        \nabla_x f(\vx,\vy)  \\
        -\nabla_y  f(\vx,\vy) 
    \end{bmatrix} +
    \begin{bmatrix}
        \partial \gI_{\gX} (\vx) \\
        \partial \gI_{\gY} (\vy)
    \end{bmatrix}.
\end{align*}
\end{lem}
\begin{proof} For any $ (\vx_1,\vy_2),(\vx_2,\vy_2)\in \gX \times \gY$,
\begin{align*}
    f(\vx_2,\vy_1) - f(\vx_1,\vy_1) &\ge \langle \nabla_x f(\vx_1,\vy_1) + \vu_1, \vx_2 - \vx_1 \rangle + \frac{\mu}{p} \Vert \vx_2 - \vx_1 \Vert^p, \quad \vu_1 \in \partial \gI_{\gX} (\vx_1); \\
    f(\vx_1,\vy_2) - f(\vx_2,\vy_2) &\ge \langle \nabla_x f(\vx_2, \vy_2) + \vu_2, \vx_1 - \vx_2 \rangle + \frac{\mu}{p} \Vert \vx_2 - \vx_1 \Vert^p, \quad \vu_2 \in \partial \gI_{\gX} (\vx_2).
\end{align*}
and 
\begin{align*}
    -f(\vx_1,\vy_2) + f(\vx_1,\vy_1) &\ge \langle -\nabla_y f(\vx_1,\vy_1) +\vv_1, \vy_2 - \vy_1 \rangle + \frac{\mu}{p} \Vert \vy_2 - \vy_1 \Vert^p, \quad \vv_1\in \partial \gI_{\gY}(\vy_1);  \\
    -f(\vx_2,\vy_1) + f(\vx_2,\vy_2) &\ge  \langle -\nabla_y f(\vx_2,\vy_2) + \vv_2, \vy_1 - \vy_2 \rangle + \frac{\mu}{p} \Vert \vy_2 - \vy_1 \Vert^p, \quad \vv_2 \in \partial \gI_{\gY} (\vy_2).
\end{align*}
Summing up the above four equations yields Eq. (\ref{eq:U-monotone}).

\end{proof}

\subsection{Proof of Lemma \ref{lem:reduction-UCUC}}

By the definition of $\tilde f$, for every $\vx \in \gX$, $\vy \in \gY$, it holds that $\vert \tilde f(\vx, \vy)  - f(\vx,\vy) \vert \le \epsilon / 3$. Then
\begin{align*}
    \max_{\vy \in \gY} f(\hat \vx, \vy) &\le \max_{\vy \in \gY} \tilde f(\hat \vx, \vy) + \epsilon / 3 \\
    \min_{\vx \in \gX} f(\vx, \hat \vy) & \ge \min_{\vx \in \gX} \tilde f(\vx, \hat \vy) - \epsilon / 3. 
\end{align*}
If $(\hat \vx, \hat \vy)$ is an $(\epsilon/3)$-solution to the regularized function $\tilde f$, \textit{i.e.}, 
\begin{align*}
    \max_{\vy \in \gY} 
    \tilde f(\vx, \vy)
     - \min_{\vx \in \gX} \tilde f(\vx, \hat \vy) \le \epsilon /3,
\end{align*}
then we can conclude that 
\begin{align*}
    \max_{\vy \in \gY} 
    f(\vx, \vy)
     - \min_{\vx \in \gX} f(\vx, \hat \vy) \le \epsilon.
\end{align*}

\subsection{Proof of Lemma \ref{lem:max-preserve-convex}}
\begin{proof}
Let $y^*(\vx) = \arg\max_{\vy \in \gY} f(\vx,\vy)$.
For any $\vx_1,\vx_2 \in \gX$, we have that
\begin{align*}
 &\quad   \Phi(\vx_1) - \Phi(\vx_2) - \langle \Phi(\vx_2), \vx_1 - \vx_2 \rangle \\
 &= f(\vx_1, \vy^*(\vx_1)) - f(\vx_2, \vy^*(\vx_2)) - \langle \nabla_x f(\vx_2, \vy^*(\vx_2), \vx_1 - \vx_2 \rangle \\
 &\ge f(\vx_1, \vy^*(\vx_2)) - f(\vx_2, \vy^*(\vx_2)) - \langle \nabla_x f(\vx_2, \vy^*(\vx_2), \vx_1 - \vx_2 \rangle \\
 &\ge \frac{\mu_x}{3} \Vert \vx_1 - \vx_2 \Vert^3,
\end{align*}
where the last step uses the fact that $f(\,\cdot\,,\vy)$ is $\mu_x$-uniformly convex.
This is exactly the $\mu_x$-uniform convexity by definition. Repeating the same analysis in variable $\vy$ shows that $\Psi(\vy)$ is $\mu_y$-uniformly concave.
\end{proof}

\subsection{Proof of Lemma \ref{lem:cont-sol-set}}

\begin{proof} 
% We first give an illustrating proof for the unconstrained case when $\gY = \sR^{d_y}$, which is similar to Lemma 4.1 \citep{chen2024finding}.
% For any $\vx_1,\vx_2 \in \gX$, we have that
% \begin{align*}
%     &\quad \mu_y \Vert \vy^*(\vx_1) - \vy^*(\vx_2) \Vert^2 \\
%     &\le \Vert \nabla_y f(\vx_1,\vy^*(\vx_2)) \Vert \\
%     &= \Vert \nabla_y f(\vx_1,\vy^*(\vx_2)) - \nabla_y f(\vx_2,\vy^*(\vx_2)) \Vert \\
%     &\le \ell \Vert \vx_1- \vx_2 \Vert,
% \end{align*}
% where the first line uses Lemma \ref{lem:UC-grad-dominant}, the second line uses $\nabla_y f(\vx_2, \vy^*(\vx_2) =0$ as $\gY = \sR^{d_y}$. 
% Below, we give the proof for general $\gY$,
The proof is similar to \citep[Lemma B.2]{lin2020near}. Let $\vx_1,\vx_2 \in \gX$, from the optimality condition, we have that
\begin{align*}
   \langle \nabla_y f(\vx_1, \vy^*(\vx_1),  \vy - \vy^*(\vx_1) \rangle &\le 0, \quad \forall \vy \in \gY; \\
   \langle \nabla_y f(\vx_2, \vy^*(\vx_2), \vy' - \vy^*(\vx_2) \rangle &\le 0,\quad \forall \vy' \in \gY.
\end{align*}
Plugging $\vy = \vy^*(\vx_2) $ and $\vy' = \vy^*(\vx_1)$ into the above inequalities and summing them up gives
\begin{align*}
    \langle \nabla_y f(\vx_1, \vy^*(\vx_1) - \nabla_y f(\vx_2, \vy^*(\vx_2)),  \vy^*(\vx_2) - \vy^*(\vx_1) \rangle  \le 0.
\end{align*}
By the $\mu_y$-uniform concavity in $\vy$, we know from of \citep[Eq. (4.2.12)]{nesterov2018lectures} that
\begin{align*}
    \langle \nabla_y f(\vx_1, \vy^*(\vx_2)) - \nabla_y f(\vx_1, \vy^*(\vx_1)),  \vy^*(\vx_2) - \vy^*(\vx_1) \rangle + \mu_y \Vert \vy^*(\vx_2) - \vy^*(\vx_1) \Vert^3  \le 0.
\end{align*}
Now we can sum up the above inequality to obtain that 
\begin{align*}
   \langle \nabla_y f(\vx_1, \vy^*(\vx_2)) - \nabla_y f(\vx_2, \vy^*(\vx_2)),  \vy^*(\vx_2) - \vy^*(\vx_1) \rangle + \mu_y \Vert \vy^*(\vx_2) - \vy^*(\vx_1) \Vert^3  \le 0.
\end{align*}
We apply the Lipschitz continuity of $\nabla_y f (\vx,\vy)$ to get
\begin{align*}
    \mu_y \Vert \vy^*(\vx_2) - \vy^*(\vx_1) \Vert^3 \le \ell \Vert \vx_1 - \vx_2 \Vert  \Vert \vy^*(\vx_2) - \vy^*(\vx_1) \Vert,
\end{align*}
which proves the continuity of $\vy^*(\vx)$. Similarly, we can also show the result for $\vx^*(\vy)$.
\end{proof}

\section{Proofs in Section \ref{subsec:main}}

Below, we show that the inexact zeroth-order and first-order oracles are easily obtainable for both the primal objective $\Phi(\vx) := \max_{\vy \in \gY} f(\vx,\vy)$ and dual objective $\Psi(\vy; \bar \vx):= \min_{\vx \in \gX} g(\vx,\vy ; \bar \vx)$, where $g(\vx,\vy; \bar \vx) := f(\vx,\vy) + \frac{\gamma}{3} \Vert \vx- \bar \vx \Vert^3$.

\begin{thm} \label{thm:get-zo-fo}
Under Assumption \ref{asm:function-lip}, \ref{asm:grad-lip}, \ref{asm:Hess-lip}, \ref{asm:UC-UC} and \ref{asm:M-min},
$\gM_{\rm min}$ can
\begin{itemize}
    \item find a point $\hat \vy$ such that $\vert f(\vx,\hat \vy) - \Phi(\vx) \vert \le \delta$ in $\gO \left( T_{\rm min}(\rho, \mu_y) \log (  D L /\delta  \right))$ iterations. 
    \item find a point $\hat \vy$ such that $\Vert \nabla f(\vx,\hat \vy) - \nabla \Phi(\vx) \Vert \le \delta$ in $\gO \left( T_{\rm min}(\rho, \mu_y) \log ( D \ell / \delta \right))$ iterations. 
    \item find a point $\hat \vx$ such that $\vert g(\hat \vx,\vy; \bar \vx) - \Psi(\vy ; \bar \vx) \vert \le \delta$ in $\gO \left( T_{\rm min}(\rho+ 2\gamma, \gamma/2) \log (  D L /\delta  \right))$ iterations. 
    \item finds a point $\hat \vx$ such that $\Vert \nabla_y g(\hat \vx,\vy; \bar \vx) - \nabla \Psi(\vy ; \bar \vx) \Vert \le \delta$ in $\gO \left( T_{\rm min}(\rho+2\gamma, \gamma/2) \log (  D \ell /\delta  \right))$ iterations. 
\end{itemize}

% and $\vert \nabla_x f(\vx,\vy) - \nabla \Phi(\vx) \vert \le \delta$ in

% and $\Psi(\vy; \bar \vx):= \min_{\vx \in \gX} g(\vx,\vy ; \bar \vx)$, where $g(\vx,\vy; \bar \vx) = f(\vx,\vy) + \frac{\gamma}{3} \Vert \vx- \bar \vx \Vert^3$. 
\end{thm}

\begin{proof}
To obtain an inexact zeroth-order oracle under 
Assumption \ref{asm:function-lip}, it suffices to find $\hat \vy$ such that $\Vert \hat \vy - \vy^*(\vx) \Vert \le \delta /L$, requires $\gM_{\rm min}$ in $\gO \left( (\rho/ \mu_y)^{2/7} \log ( D L / (\delta \mu_y))  \right)$ iterations. Similarly, by Danskin's theorem $\nabla \Phi(\vx) = \nabla_x f(\vx,\vy^*(\vx))$, to obtain an inexact first-order oracle under Assumption \ref{asm:grad-lip}, it suffices to find $\hat \vy$ such that $\Vert \hat \vy - \vy^*(\vx) \Vert \le \delta/ \ell$. This proves the first two claims. And the last two claims are similar by noting that $g(\vx,\vy;\bar \vx)$ is $(\gamma/2)$-uniformly convex in $\vx$ and has $(\rho+2\gamma)$-Lipschitz continuous Hessians.
\end{proof}

\subsection{Proof of Theorem \ref{thm:outer} }

Let $(\vx^*,\vy^*) = \arg \min_{\vx \in \gX} \max_{\vy \in \gY} f(\vx,\vy)$. \textit{Line 1} in Algorithm \ref{alg:Minimax-AIPE} finds $\Vert \hat \vx - \vx^* \Vert \le \zeta_1$ in $ \gO( (D_x^3 \gamma / \mu_x)^{2/7} \log (D / \zeta_1)) $ calls of the subroutine Algorithm \ref{alg:Minimax-AIPE-mid}. \textit{Line 2} in Algorithm \ref{alg:Minimax-AIPE} finds 
 $\Vert \hat \vy - \vy^*(\hat \vx) \Vert \le \zeta_1$, where $\vy^*(\vx) = \arg \max_{\vy \in \gY} f(\vx,\vy)$. Then
\begin{align*}
    &\quad \Vert \hat \vy - \vy^* \Vert \le \Vert \hat \vy - \vy^*(\hat \vx) \Vert + \frac{\ell^{1/2}}{\mu_y^{1/2}} \sqrt{\Vert \hat \vx - \vx^* \Vert}, 
\end{align*}
where we use Lemma \ref{lem:cont-sol-set} and $\vy^*(\vx^*) = \vx^*$ in the above inequality. It means
\begin{align*}
    \Vert (\hat \vx,\hat \vy) - (\vx^*,\vy^*) \Vert \le 3 \sqrt{\frac{\ell \zeta_1}{\mu_y}}
\end{align*}
Then from Lemma \ref{lem:eg-grad} we know \textit{Line 3} in Algorithm \ref{alg:Minimax-AIPE} outputs $\vz^{\rm out} =  (\vx^{\rm out},\vy^{\rm out})  $ such that 
\begin{align*}
    \left\Vert 
         \mF(\vz^{\rm out}) + \vc^{\rm out}
        \right \Vert \le 7  \ell \sqrt{\frac{3\ell \zeta_1}{\mu_y}}
    % \begin{bmatrix}
    %      \\
    %     -\partial 
    % \end{bmatrix}.
\end{align*}
for some $(\vu^{\rm out}, \vv^{\rm out}) = \vc^{\rm out} \in  \partial \gI_{\gZ}(\vz^{\rm out})$.
Using the convex-concavity, we have that
\begin{align*}
 &\quad f(\vx^{\rm out},\vy^*(\vx^{\rm out})) - f(\vx^*(\vy^{\rm out}),\vy^{\rm out}) \\
 &= f(\vx^{\rm out},\vy^*(\vx^{\rm out})) -  f(\vx^{\rm out},\vy^{\rm out}) + f(\vx^{\rm out},\vy^{\rm out}) -  f(\vx^*(\vy^{\rm out}),\vy^{\rm out}) \\
 &\le \langle \nabla_x f(\vx^{\rm out},\vy^{\rm out}) + \vu^{\rm out}, \vx^*(\vy^{\rm out}) -\vx^{\rm out} \rangle \\
 &\quad + \langle - \nabla_y f(\vx^{\rm out},\vy^{\rm out}) + \vv^{\rm out},  \vy^*(\vx^{\rm out}) - \vy^{\rm out} \rangle.
\end{align*}
Finally, the Cauchy–Schwarz inequality tells 
\begin{align*}
&\quad f(\vx^{\rm out},\vy^*(\vx^{\rm out})) - f(\vx^*(\vy^{\rm out}),\vy^{\rm out}) \\
&\le \left\Vert 
    \begin{bmatrix}
          \nabla_x f(\vx^{\rm out}, \vy^{\rm out}) + \vu^{\rm out} \\
        - \nabla_y f(\vx^{\rm out}, \vy^{\rm out}) + \vv^{\rm out}
    \end{bmatrix}
        \right \Vert \le 7  \ell D \sqrt{\frac{3\ell \zeta_1}{\mu_y}}.
\end{align*}

\subsection{Proof of Theorem \ref{thm:middle}}
Let $ (\vx^*(\bar \vx), \vy^*(\bar \vx)) = \arg \min_{\vx \in \gX} \max_{\vy \in \gY} g(\vx,\vy; \bar \vx) $ and $\Psi(\vy; \bar \vx) = \min_{\vx \in \gX} f(\vx,\vy)$. \textit{Line 1} in Algorithm \ref{alg:Minimax-AIPE-mid} finds 
$ \Vert \hat \vy - \vy^*(\bar \vx) \Vert \le \zeta_2$ in $\gO( (D_y^3 \gamma / \mu_y )^{2/7} \log (D/ \zeta_2))$ calls of the subroutine Algorithm \ref{alg:Minimax-AIPE-inner}, where $\vy^*(\bar \vx) = \arg \max_{\vy \in \gY} \Psi(\vy; \bar \vx)$. \textit{Line 2} in Algorithm \ref{alg:Minimax-AIPE-mid} finds $\hat \vx$ such that $\Vert \hat \vx - \vx^*(\hat \vy; \bar \vx) \Vert \le \zeta_2$.
By Lemma \ref{lem:eg-grad}, \textit{Line 3} in Algorithm \ref{alg:Minimax-AIPE-mid} further ensures that $\Vert \nabla_x g(\vx^{\rm out}, \hat \vy; \bar \vx) + \vu^{\rm out} \Vert \le 
6 (\ell + 2 \gamma D) \zeta_2 $ for some $\vu^{\rm out} \in \partial \gI_{\gX}(\vx^{\rm out})$. 
%By Lemma \ref{lem:U-monotone} and the Cauchy-Schwarz inequality we have $\Vert \vx - \vx^*(\vy; \bar \vx) \Vert \le 2 \sqrt{\frac{\gamma+2 \rho}{\mu_x}} \zeta_2$.
Invoking Lemma \ref{lem:cont-sol-set}, we have that
\begin{align} \label{eq:cont-mid}
\begin{split}
     \Vert \vx^*(\vy_1; \bar \vx) - \vx^*(\vy_2; \bar \vx) \Vert^2 &\le \frac{\ell +2 \gamma D}{\gamma/2} \Vert \vy_1 - \vy_2 \Vert, \quad \forall \vy_1,\vy_2 \in \gY; \\
    \Vert \vy^*(\vx_1; \bar \vx) - \vy^*(\vx_2; \bar \vx) \Vert^2 &\le \frac{\ell}{\mu_y} \Vert \vx_1 - \vx_2 \Vert, \quad \forall \vx_1,\vx_2 \in \gX,
\end{split}
    % \Vert  \vx^*(\vy_1; \bar\vx , \bar \vy) - \vx^*(\vy_2; \bar\vx , \bar \vy) \Vert^2 &\le \frac{\ell + 2 \rho D}{\rho} \Vert \vy_1 - \vy_2 \Vert, \quad \forall \vy_1, \vy_2 \in \gY; \\
\end{align}
where $\vx^*(\vy; \bar \vx) = \arg \min_{\vx \in \gX} g(\vx,\vy; \bar \vx)$ and $\vy^*(\vx; \bar \vx) =\arg \max_{\vy \in \gY} g(\vx,\vy; \bar\vx) $.
Then 
\begin{align} \label{eq:xxx-plug}
\begin{split}
\Vert \vx^{\rm out} - \vx^*(\bar \vx) \Vert &\le \Vert \vx^{\rm out} - \vx^*(\hat \vy; \bar \vx) \Vert +  \frac{(\ell + 2 \gamma D )^{1/2}}{(\gamma/2)^{1/2}} \sqrt{ \Vert \hat \vy - \vy^*(\bar \vx) \Vert} \\
&\le \frac{1}{(\gamma/2)^{1/2}} \sqrt{ \Vert \nabla_x g(\vx^{\rm out},\hat \vy; \bar \vx) + \vu^{\rm out} \Vert } + \frac{(\ell + 2 \gamma D )^{1/2}}{(\gamma/2)^{1/2}} \sqrt{ \Vert \hat \vy - \vy^*(\bar \vx) \Vert},
\end{split}
\end{align}
where we use (\ref{eq:cont-mid}) and $\vx^*(\vy^*(\bar \vx); \bar \vx) = \vx^*(\bar \vx)$ in the first inequality, Lemma \ref{lem:U-monotone} and the Cauchy-Schwarz inequality in the second one. Using a similar analysis, we can further show that
\begin{align*}
    &\quad \left \Vert \nabla \Phi(\vx^{\rm out}) + \gamma \Vert \vx^{\rm out} - \bar \vx \Vert (\vx^{\rm out} - \bar \vx) + \vu^{\rm out}  \right \Vert = \Vert \nabla_x g(\vx^{\rm out}, \vy^*(\vx^{\rm out}; \bar \vx) ;\bar \vx) + \vu^{\rm out}  \Vert \\
    &\le     \Vert \nabla_x g(\vx^{\rm out}, \vy^*(\bar \vx) ;\bar \vx) + \vu^{\rm out}  \Vert + \frac{(\ell + 2 \rho D) \ell^{1/2}}{\mu_y^{1/2}} \sqrt{\Vert \vx^{\rm out}  - \vx^*(\bar \vx) \Vert} \\
    &\le \Vert \nabla_x g(\vx^{\rm out}, \hat \vy; \bar \vx) + \vu^{\rm out} \Vert + (\ell +2 \rho D) \Vert \hat \vy - \vy^*(\bar \vx) \Vert + \frac{(\ell + 2 \rho D) \ell^{1/2}}{\mu_y^{1/2}} \sqrt{\Vert \vx^{\rm out}  - \vx^*(\bar \vx) \Vert},
\end{align*}
where $ \vy^*(\vx; \bar \vx) = \arg \max_{\vy \in \gY} g(\vx,\vy; \bar \vx) $. 
Now we plug Eq. (\ref{eq:xxx-plug}) into the above inequality to get
\begin{align} \label{eq:prox-x-1}
\begin{split}
&\quad   \Vert \nabla_x g(\vx^{\rm out}, \vy^*(\vx^{\rm out}; \bar \vx) ;\bar \vx) + \vu^{\rm out}  \Vert \\
&\le 6 (\ell + 2 \gamma D) \zeta_2 +  \frac{(\ell + 2 \rho D) \ell^{1/2}}{\mu_y^{1/2} (\gamma/2)^{1/4}} \left( 
   6 (\ell+2\gamma D) \zeta_2
    \right)^{1/4} \\
    &\quad + (\ell + 2\rho D) \zeta_2 +  \frac{(\ell + 2 \rho D) \ell^{1/2} (\ell + 2 \gamma D)^{1/4}}{\mu_y^{1/2} (\gamma/2)^{1/4}} (\zeta_2)^{1/2},
\end{split}
\end{align}
It means that Algorithm \ref{alg:Minimax-AIPE-mid} can successfully implement a $(\delta,\gamma)$-proximal oracle that satisfies Assumption \ref{asm:prox-Phi} by setting $\zeta_2 = \Omega(1/ {\rm poly} (\rho, \ell, D,\gamma, \mu_x^{-1}, \mu_y^{-1}, \zeta_1^{-1} ))$. Finally, the complexity of obtaining inexact zeroth-order and first-order oracles for $\Phi(\vx)$ is given in Theorem \ref{thm:get-zo-fo}.

% Finally, \textit{Line 4} of Algorithm \ref{alg:Minimax-AIPE-mid} finds $\vy' $ such that $\Vert \vy' - \vy^*(\vx) \Vert \le \zeta_2$, where $\vy^*(\vx) = \arg \max_{\vy \in \gY} f(\vx,\vy) $. Then
% \begin{align} \label{eq:prox-x-2}
%     \Vert \nabla \Phi(\vx) -  \nabla_x f(\vx,\vy') \Vert \le \ell \Vert \vy^*(\vx) - \vy' \Vert \le \ell \zeta_2. 
% \end{align}
% From (\ref{eq:prox-x-1}) and (\ref{eq:prox-x-2}) 

% Therefore, a sufficiently small $\zeta_2$ can fulfill the inexact condition of running APPM-2 on $P(\vx)= \max_{\vy \in \gY} f(\vx,\vy)$.
% Then we can find $\Vert \vx - \vx^* \Vert \le \zeta_3$ in $\gO( (\rho / \mu_x)^{2/7} \log (D/ \zeta_3) )$ iterations. 

\subsection{Proof of Theorem \ref{thm:ANPE-inner}}
Let $ (\vx^*(\bar \vx, \bar \vy), \vy^*(\bar \vx, \bar \vy)) = \arg \min_{\vx \in \gX} \max_{\vy \in \gY} h(\vx,\vy; \bar \vx, \bar \vy)$. 
By Theorem \ref{thm:NPE-restart}, we know that \textit{Line 1} of Algorithm \ref{alg:Minimax-AIPE-inner} can find $\Vert (\hat \vx,\hat \vy) - (\vx^*(\bar \vx, \bar \vy), \vy^*(\bar \vx, \bar \vy)) \Vert \le \zeta_3$ in $\gO( ( (\rho + 2\gamma) / \gamma)^{2/3}\log(D/\zeta_3))$ iterations.
And by Lemma \ref{lem:eg-grad}, \textit{Line 2} of Algorithm \ref{alg:Minimax-AIPE-inner} guarantees that 
\begin{align*}
   \left \Vert
   \begin{bmatrix}
       \nabla_x h(\vx^{\rm out},\vy^{\rm out}; \bar \vx, \bar \vy) + \vu^{\rm out} \\
       - \nabla_y h(\vx^{\rm out},\vy^{\rm out}; \bar \vx, \bar \vy) + \vv^{\rm out}
   \end{bmatrix}
   \right \Vert \le 6 (\ell + 2\gamma D) \zeta_3
\end{align*}
for some $\vu^{\rm out} \in \partial \gI_{\gX}(\vx)$ and $\vv^{\rm out} \in \partial \gI_{\gY} (\vy)$. Invoking Lemma \ref{lem:cont-sol-set}, we know that 
\begin{align*}
    \Vert  \vx^*(\vy_1; \bar\vx , \bar \vy) - \vx^*(\vy_2; \bar\vx , \bar \vy) \Vert^2 &\le \frac{\ell + 2 \gamma D}{\gamma /2 } \Vert \vy_1 - \vy_2 \Vert, \quad \forall \vy_1, \vy_2 \in \gY,
    % \\
    % \Vert \vx^*(\vy_1; \bar \vx) - \vx^*(\vy_2; \bar \vx) \Vert^2 &\le \frac{\ell +2 \rho D}{\rho} \Vert \vy_1 - \vy_2 \Vert, \quad \forall \vy_1,\vy_2 \in \gY; \\
    % \Vert \vy^*(\vx_1; \bar \vx) - \vy^*(\vx_2; \bar \vx) \Vert^2 &\le \frac{\ell}{\mu_y} \Vert \vx_1 - \vx_2 \Vert, \quad \forall \vx_1,\vx_2 \in \gX.
\end{align*}
where $ \vx^*(\vy; \bar\vx , \bar \vy) = \arg \min_{\vx \in \gX} h(\vx,\vy; \bar \vx, \bar \vy)$.
Then
\begin{align} \label{eq:prox-xy-1}
\begin{split}
    &\quad \left \Vert \nabla \Psi(\vy^{\rm out}; \bar \vx) + \gamma \Vert \vy^{\rm out} - \bar \vy \Vert ( \vy^{\rm out} - \bar \vy) + \vv^{\rm out} \right \Vert = \Vert \nabla_y h(\vx^*(\vy^{\rm out}; \bar \vx, \bar \vy),\vy^{\rm out}; \bar \vx, \bar \vy) + \vv^{\rm out} \Vert \\
    &\le  \Vert \nabla_y h(\vx^*(\bar \vx, \bar \vy), \vy^{\rm out};\bar \vx, \bar \vy) + \vv^{\rm out} \Vert +  \frac{(\ell +2 \gamma D)^{3/2}}{(\gamma/2)^{1/2}} \Vert \vy^{\rm out} - \vy^*(\bar \vx, \bar \vy) \Vert \\
    &\le \Vert \nabla_y h(\vx^{\rm out}, \vy^{\rm out}; \bar \vx, \bar \vy) +\vv^{\rm out} \Vert + (\ell +2 \gamma D) \Vert \vx^{\rm out} - \vx^*(\bar \vx, \bar \vy) \Vert \\
    &\quad + \frac{(\ell +2 \gamma D)^{3/2}}{(\gamma/2)^{1/2}} \sqrt{\Vert \vy^{\rm out} - \vy^*(\bar \vx, \bar \vy) \Vert} \\
    &\le 6 (\ell + 2 \gamma D) \zeta_3
    + \frac{\ell + 2\gamma D}{(\gamma/2)^{1/2}}   \left( 6 (\ell+2 \gamma D) \zeta_3 \right)^{1/2}  +\frac{(\ell +2 \gamma D)^{3/2}}{(\gamma/2)^{3/4}} \left( 6 (\ell+2 \gamma D) \zeta_3 \right)^{1/4},
\end{split}
\end{align}
where in the last line we use Lemma \ref{lem:U-monotone} as well as the Cauchy-Schwartz inequality to upper bound the distance to saddle point with gradient norm. The above inequality means that Algorithm \ref{alg:Minimax-AIPE-inner} can successfully implement a $(\delta,\gamma)$-proximal oracle that satisfies Assumption \ref{asm:prox-Psi} by setting $\zeta_2 = \Omega(1/ {\rm poly} (\rho, \ell, D,\gamma, \mu_y^{-1}, \zeta_2^{-1} ))$.
Finally, the complexity of obtaining inexact zeroth-order and first-order oracles for $\Psi(\vy; \bar \vx)$ is given in Theorem \ref{thm:get-zo-fo}.

\section{Proofs in Section \ref{subsec:acc-practice}}

\subsection{Guarantee of the NPE Subroutine}

\citet{monteiro2012iteration} proposed the Newton Proximal Extragradient (NPE) method for monotone variational inequalities, which can find an $\epsilon$-solution with $\gO(\epsilon^{-2/3})$ second-order oracle calls. As noted in \citep{bullins2022higher,huang2022approximation,lin2022explicit,lin2022perseus,adil2022optimal}, the procedure of NPE can be simplified by using the cubic regularized Newton oracle. We present the simplified version in Algorithm \ref{alg:NPE}. 
It is known \citep{monteiro2012iteration,adil2022optimal,bullins2022higher} that Algorithm \ref{alg:NPE} can provably find an $\epsilon$-solution to the variational inequality problem induced by a monotone operator $\mF$ in $\gO(\epsilon^{-2/3})$ iterations. And applying the restart strategy on it (Algorithm \ref{alg:NPE-restart}) can solve a $\mu$-strongly monotone variational inequality problem in $\gO( (\gamma/\mu)^{2/3} \log \epsilon^{-1}   )$ iteration complexity, as stated in the following theorem.        

\begin{algorithm*}[t]  
\caption{NPE$(\vz_0, T, \gamma) $
}\label{alg:NPE}
\begin{algorithmic}[1] 
\renewcommand{\algorithmicrequire}{ \textbf{Input:}}
%\REQUIRE Function $f$ and its MS oracle $\gA_{\rm MS}$; Initial $\vz_0$; Iterations $T$; Parameter $\alpha$. \\
\FOR{$t = 0,\cdots, T-1$}  
\STATE \quad $\vz_{t+/2}, \vu_{t+1/2} = {\rm CRN}( \vz_t, \gamma)$ \\
\STATE \quad $\eta_t = \frac{1}{2 \gamma \Vert \vz_t - \vz_{t+1/2} \Vert}$ \\
\STATE \quad $\vz_{t+1} = \arg \min_{\vz \in \gZ} \left\{ \langle \eta_t \mF(\vz_{t+1/2}), \vz - \vz_t \rangle + \frac{1}{2} \Vert \vz - \vz_t \Vert^2  \right\} $.
\ENDFOR \\
\RETURN $z_{\rm out} = \frac{1}{\sum_{t=0}^{T-1} \eta_t} \sum_{t=0}^{T-1} \eta_t \vz_{t+1/2}$ \\
\end{algorithmic}
\end{algorithm*}

\begin{algorithm*}[t]  
\caption{NPE-restart$(\vz_0, T, \gamma, S) $
}\label{alg:NPE-restart}
\begin{algorithmic}[1] 
\renewcommand{\algorithmicrequire}{ \textbf{Input:}}
%\REQUIRE Function $f$ and its MS oracle $\gA_{\rm MS}$; Initial $\vz_0$; Iterations $T$; Parameter $\alpha$. \\
\STATE $\vz^{(0)} = \vz_0$ \\
\FOR{$s=0,\cdots,S-1$}
\STATE \quad $\vz^{(s+1)} =\text{NPE}(\vz^{(s)}, T, \gamma) $ \\
\ENDFOR \\
\RETURN $\vz^{(S)}$
\end{algorithmic}
\end{algorithm*}
% The CRN oracle has the following theoretical guarantees.        
\begin{thm}[NPE-restart] \label{thm:NPE-restart}
Under Assumption \ref{asm:UC-UC} with $\mu_x = \mu_y = \mu$ and Assumption \ref{asm:Hess-lip}, running Algorithm \ref{alg:NPE-restart} with $\gamma = 2 \rho$ and $T = \gO(  (\gamma/\mu)^{2/3})$ and $S = \gO(\log (d_0/\epsilon))$ returns a point $\vz^{(S)}$ such that $\Vert \vz^{(S)} - \vz^* \Vert \le \epsilon$, where $\vz^*$ is the unique solution to Problem (\ref{prob:VI}) and $d_0 = \Vert \vz_0 - \vz^* \Vert$.
\end{thm}

\begin{proof}
By Theorem \ref{thm:NPE}, each epoch of Algorithm \ref{alg:NPE-restart} ensures $\Vert \vz^{(s+1)}- \vz^* \Vert \le \frac{1}{2} \Vert \vz^{(s)}- \vz^* \Vert$ if setting $T = \gO\left((\gamma/ \mu)^{2/3}\right)$. And therefore 
the algorithm finds a point $\vz^{(S)}$ such that $\Vert \vz^{(S)} - \vz^* \Vert \le \epsilon$ in $S = \left \lceil \log_2 (d_0 /\epsilon) \right \rceil$ epochs.
% We can all this property the uniform monotonicity of operator $\mF(\vz)$.  The goal of each epoch in Algorithm \ref{alg:NPE-restart} is to ensure that
% \begin{align*}
%     \frac{2\mu}{3} \Vert \vz^{(s+1)} - \vz^* \Vert^3 \le \frac{\mu}{3} \Vert \vz^{(s)} - \vz^* \Vert^3 := \epsilon^{(s)}.
% \end{align*}
% By (\ref{eq:U-monotone}) and (\ref{eq:regret-NPE}), in the $s$-th epoch it suffices to run Algorithm \ref{alg:NPE}  with precision $\epsilon^{(s)}$. Then we invoke Theorem \ref{thm:NPE-restart} in each epoch.
\end{proof}

\begin{thm} \label{thm:NPE}
Under the same setting as Theorem \ref{thm:NPE-restart}, running Algorithm \ref{alg:NPE} with $\gamma = 2\rho$ outputs a point $\vz_{\rm out}$ such that $\Vert \vz_{\rm out} - \vz^* \Vert \le \epsilon$ in $T = \gO\left( \left( \frac{\gamma d_0^3}{\mu \epsilon^3 } \right)^{2/3} \right)$ iterations, where $\vz^*$ is the unique solution to Problem (\ref{prob:VI}) and $d_0 = \Vert \vz_0 - \vz^* \Vert$.
\end{thm}

\begin{proof}
It is known \citep{monteiro2012iteration,adil2022optimal,bullins2022higher} that Algorithm \ref{alg:NPE} ensures 
\begin{align} \label{eq:regret-NPE}
    {\rm Reget} :=\frac{1}{\sum_{t=0}^{T-1} \eta_t} \sum_{t=0}^{T-1} \eta_t \langle  \mF(\vz_{t+1/2}), \vz_{t+1/2} - \vz^* \rangle = \gO\left( \frac{\gamma \Vert \vz_0 - \vz^* \Vert^3}{T^{3/2}} \right).
\end{align} 
We further use Lemma \ref{lem:U-monotone}, the convexity of $\Vert \cdot \Vert^3$ and Jensen's inequality to derive that
\begin{align*}
    \Vert \vz_{\rm out} - \vz^* \Vert^3 \le \frac{1}{\sum_{t=0}^{T-1} \eta_t} \sum_{t=0}^{T-1} \eta_t \Vert \vz_{t+1/2} - \vz^*\Vert^3 \le \frac{3}{2 \mu} {\rm Regret}, 
\end{align*}
which leads to the result.

\end{proof}

\newpage
\subsection{Guarantee of the LEN Subroutine}  \label{apx:LEN-const}

\begin{algorithm*}[t]  
\caption{LEN$(\vz_0, T, \gamma) $
}\label{alg:LEN}
\begin{algorithmic}[1] 
\renewcommand{\algorithmicrequire}{ \textbf{Input:}}
%\REQUIRE Function $f$ and its MS oracle $\gA_{\rm MS}$; Initial $\vz_0$; Iterations $T$; Parameter $\alpha$. \\
\FOR{$t = 0,\cdots, T-1$}  
\STATE \quad $\vz_{t+/2}, \vu_{t+1/2} = {\rm LazyCRN}( \vz_t,  \vz_{\pi(t)},  \gamma)$ \\
\STATE \quad $\eta_t = \frac{1}{2 \gamma \Vert \vz_t - \vz_{t+1/2} \Vert}$ \\
\STATE \quad $\vz_{t+1} = \arg \min_{\vz \in \gZ} \left\{ \langle \eta_t \mF(\vz_{t+1/2}), \vz - \vz_t \rangle + \frac{1}{2} \Vert \vz - \vz_t \Vert^2  \right\} $.
\ENDFOR \\
\RETURN $z_{\rm out} = \frac{1}{\sum_{t=0}^{T-1} \eta_t} \sum_{t=0}^{T-1} \eta_t \vz_{t+1/2}$ \\
\end{algorithmic}
\end{algorithm*}

\begin{algorithm*}[t]  
\caption{LEN-restart$(\vz_0, T, \gamma, S) $
}\label{alg:LEN-restart}
\begin{algorithmic}[1] 
\renewcommand{\algorithmicrequire}{ \textbf{Input:}}
%\REQUIRE Function $f$ and its MS oracle $\gA_{\rm MS}$; Initial $\vz_0$; Iterations $T$; Parameter $\alpha$. \\
\STATE $\vz^{(0)} = \vz_0$ \\
\FOR{$s=0,\cdots,S-1$}
\STATE \quad $\vz^{(s+1)} =\text{LEN}(\vz^{(s)}, T, \gamma) $ \\
\ENDFOR \\
\RETURN $\vz^{(S)}$
\end{algorithmic}
\end{algorithm*}

This section briefly reviews the results in the recent work \citep{chen2024second,chen2025computationally}. \citet{chen2024second} proposed the Lazy Extra Newton (LEN) method as the lazy version of NPE, and in the subsequent work, \citet{chen2025computationally} proposed the accelerated LEN (A-LEN) as the lazy version of A-NPE. Instead of using the CRN oracle, they use the following lazy CRN oracle proposed by \citet{doikov2023second} to further reduce the computational complexity of NPE and A-NPE. 
\begin{dfn}
    A lazy CRN oracle for Problem (\ref{prob:main}) takes the query point $\bar \vz \in \gZ$, the snapshot point $\vz_{\rm ss}$, and the regularization parameter $\gamma>0$ as inputs, and returns $(\vz,\vu) = {\rm CRN}(\bar \vz, \vz_{\rm ss}, \gamma)$ satisfies:
\begin{align*}
    \langle \mF(\bar \vz) + \nabla \mF(\vz_{\rm ss})(\vz - \bar \vz) + \frac{\gamma}{2} \Vert \vz - \bar \vz \Vert (\vz - \bar \vz), \vz' - \vz  \rangle \ge 0, ~\forall \vz' \in \gZ; \\
    \vu =  -\left(\mF(\bar\vz) + \nabla \mF(\vz_{\rm ss}) (\vz - \bar \vz) + \frac{\gamma}{2} \Vert \vz - \bar \vz \Vert 
    (\vz- \bar \vz)  \right) \in 
    \begin{bmatrix}
        \partial \gI_{\gX} (\vx) \\
        -\partial \gI_{\gY}(\vy)
    \end{bmatrix}.
\end{align*}
% where $\mF(\vz) = \begin{bmatrix}
%         \nabla_x f(\vx,\vy) \\
%         - \nabla_y f(\vx,\vy)
% \end{bmatrix}$. 
% Particularly, 
% for minimization problem $\min_{\vz\in\gZ}f(\vz)$ we have
% \begin{align*}
%     \vz = \arg & \min_{\vz' \in \gZ}  \left\{ 
%     f(\vz') + \langle \nabla f(\bar \vz), \vz' - \bar \vz \rangle + \frac{1}{2} \langle \nabla^2 f(\vz_{\rm ss}) (\vz' - \bar \vz), \vz' - \bar \vz \rangle + \frac{\gamma}{6} \Vert \vz' - \bar \vz \Vert^3
%     \right\}; \\
%     \vu &= - ( \nabla f(\bar \vz) + \nabla^2 f(\vz_{\rm ss}) (\vz - \bar \vz) + \frac{\gamma}{2} \Vert \vz - \bar \vz \Vert (\vz-  \bar \vz) ) \in \partial \gI_{\gZ}(\vz).
% \end{align*}
\end{dfn}

Using their result, we know that  $T_{\rm saddle}(\rho, \mu) = m+ m^{2/3} (\rho / \mu)^{2/3}$ and $T_{\rm min}(\rho, \mu) = m+ m^{5/7} (\rho/ \mu)^{2/7}$. 
A small difference is that the work \citep{chen2024second,chen2025computationally} only analyzed the unconstrained case ($\gX = \sR^{d_x}$ and $\gY = \sR^{d_y}$), but we need the results for the constraints sets $\gX$ and $\gY$ here. We remark that essentially all the analysis in the work \citep{chen2024second,chen2025computationally} holds under the constrained setting.  Below, we show the convergence of LEN-restart (Algorithm \ref{alg:LEN-restart}), which invokes LEN (Algorithm \ref{alg:LEN}) as a subroutine. In Algorithm \ref{alg:LEN}, we use the notation $\pi(t) = t - t \mod m$.
\begin{thm}[LEN-restart] \label{thm:LEN-restart}
Under Assumption \ref{asm:UC-UC} with $\mu_x = \mu_y = \mu$ and Assumption \ref{asm:Hess-lip}, running Algorithm \ref{alg:NPE-restart} with $\gamma = \gO(m \rho)$ and $T = \gO( m+ (\gamma/\mu)^{2/3})$ and $S = \gO(\log (d_0/\epsilon))$ returns a point $\vz^{(S)}$ such that $\Vert \vz^{(S)} - \vz^* \Vert \le \epsilon$, where $\vz^*$ is the unique solution to Problem (\ref{prob:VI}) and $d_0 = \Vert \vz_0 - \vz^* \Vert$.
 
\end{thm}

\begin{proof}
As shown in Theorem \ref{thm:NPE}, the restart scheme can transform the convergence under the convex-concave setting to the convergence under the uniformly-convex-uniformly-concave setting in a black-box manner. Therefore, we only need to show the convergence of LEN (Algorithm \ref{alg:LEN}) under the convex-concave setting. The proof is essentially the same as \citep[Theorem 4.1]{chen2024second}, except here we consider the constrained case. Below, we show that all the proofs in \citep[Theorem 4.1]{chen2024second} also hold under the constrained case. 

Using the first-order optimality condition in the extragradient step, we have that
\begin{align} \label{eq:opt-extra}
    0 \le \langle \eta_t \mF(\vz_{t+1/2}) + \vz_{t+1} - \vz_t, \vz - \vz_{t+1} \rangle, \quad \forall \vz \in \gZ.
\end{align}
Using the first-order optimality condition in the lazy Newton step, we have that
\begin{align}\label{eq:opt-lazy-Newton}
    0 \le \langle \eta_t \tilde \mF(\vz_{t+1/2}) + \vz_{t+1/2} - \vz_t , \vz - \vz_{t+1/2} ), \quad \forall \vz \in \gZ,
\end{align}
where $\tilde \mF(\vz_{t+1/2}) = \mF(\vz_{t}) + \nabla \mF(\vz_{\pi(t)}) (\vz_{t+1/2} - \vz_t)$. 
They together imply that
\begin{align} \label{eq:eg-1}
\begin{split}
     &\quad \eta_t \langle \mF(\vz_{t+1/2}), \vz_{t+1/2} - \vz \rangle \\
    &= \eta_t \langle \mF(\vz_{t+1/2}), \vz_{t+1} - \vz \rangle + \eta_t \langle \mF(\vz_{t+1/2}), \vz_{t+1/2} - \vz_{t+1}   \rangle \\ 
    &= \eta_t \langle \mF(\vz_{t+1/2}), \vz_{t+1} - \vz \rangle + \eta_t \langle \tilde \mF(\vz_{t+1/2}), \vz_{t+1/2} - \vz_{t+1}   \rangle \\
    &\quad + \eta_t \langle \mF(\vz_{t+1/2}) - \tilde \mF(\vz_{t+1/2}), \vz_{t+1/2} - \vz_{t+1} \rangle  \\
    &\le \langle \vz_{t+1} - \vz_t, \vz - \vz_{t+1} \rangle + \langle  \vz_{t+1/2} - \vz_t, \vz_{t+1} - \vz_{t+1/2}  \rangle \\
    &\quad + \eta_t \langle \mF(\vz_{t+1/2}) - \tilde \mF(\vz_{t+1/2}), \vz_{t+1/2} - \vz_{t+1} \rangle.
\end{split}
\end{align}
Next, using identity $\langle \va, \vb \rangle =  \frac{1}{2} \Vert \va + \vb \Vert^2 - \frac{1}{2} \Vert \va \Vert^2 + \frac{1}{2} \Vert \vb \Vert^2$ and Young's inequality yields
\begin{align*}
    &\quad \eta_t \langle \mF(\vz_{t+1/2}), \vz_{t+1/2} - \vz \rangle \\
    &\le \frac{1}{2} \Vert \vz - \vz_t \Vert^2 \bcancel{- \frac{1}{2} \Vert \vz_{t+1} - \vz_t \Vert^2} - \frac{1}{2} \Vert \vz - \vz_{t+1} \Vert^2 \\
    &\quad + \bcancel{\frac{1}{2} \Vert \vz_{t+1} - \vz_t \Vert^2} - \frac{1}{2} \Vert \vz_{t+1/2} - \vz_t \Vert^2 - \frac{1}{2} \Vert \vz_{t+1} - \vz_{t+1/2} \Vert^2 \\
    &\quad + \underbrace{\eta_t^2 \Vert \mF(\vz_{t+1/2}) - \tilde \mF(\vz_{t+1/2}) \Vert^2}_{(*)} + \frac{1}{4} \Vert \vz_{t+1} - \vz_{t+1/2} \Vert^2.
    %  - \frac{\eta_t}{2} \Vert \Vert 
    % + \frac{\eta_t}{2} \Vert \vz - \vz_{t+1} \Vert^2 + 
    % \frac{\eta_t}{2} \Vert \vz - \vz_t \Vert^2 \bcancel{-\frac{\eta_t}{2} \Vert \vz_{t+1} - \vz_t \Vert^2}  \\
    % &\quad + \frac{\eta_t}{2} \Vert \vz_{t+1/2} - \vz_t \Vert^2 + \frac{\eta_t}{2} \Vert \vz_{t+1/2} - \vz_{t+1} \Vert^2 \bcancel {- \frac{\eta_t}{2} \Vert \vz_{t} - \vz_{t+1} \Vert^2} \\
\end{align*}
We can upper bound (*) using Young's inequality and Assumption \ref{asm:Hess-lip} as
\begin{align*}
    &\quad \eta_t^2 \Vert \mF(\vz_{t+1/2}) - \tilde \mF(\vz_{t+1/2}) \Vert^2 \\
    &\le 2 \eta_t^2 \Vert \mF(\vz_{t+1/2}) - \mF(\vz_t) - \nabla \mF(\vz_t) (\vz_{t+1/2} - \vz_t) \Vert^2 \\
    &\quad + 2 \eta_t^2 \Vert (\nabla \mF(\vz_t) - \nabla \mF(\vz_{\pi(t)}) (\vz_{t+1/2} - \vz_t) \Vert^2 \\
    &\le \frac{\eta_t^2 \rho^2}{2} \Vert \vz_{t+1/2} - \vz_t \Vert^4  + 2 \eta_t^2 \rho^2 \Vert \vz_t - \vz_{\pi(t)} \Vert^2 \Vert \vz_{t+1/2} - \vz_t \Vert^2  \\
    &= \frac{\rho^2}{8 \gamma^2} \Vert \vz_{t+1/2} - \vz_t \Vert^2 + \frac{\rho^2}{2\gamma^2} \Vert \vz_t - \vz_{\pi(t)} \Vert^2.
\end{align*}
Finally, it leads to 
\begin{align*}
    &\quad \eta_t \langle \mF(\vz_{t+1/2}), \vz_{t+1/2} - \vz \rangle \\
    &\le \frac{1}{2} \Vert \vz - \vz_t \Vert^2 - \frac{1}{2} \Vert \vz - \vz_{t+1} \Vert^2 - \frac{1}{4} \Vert \vz_{t+1} - \vz_{t+1/2} \Vert^2 \\
    &\quad - \left(\frac{1}{2} - \frac{\rho^2}{8 \gamma^2} \right) \Vert \vz_{t+1/2} - \vz_t \Vert^2 + \frac{\rho^2}{2\gamma^2} \Vert \vz_t - \vz_{\pi(t)} \Vert^2.
\end{align*}
This matches \citep[Lemma 4.1]{chen2024second} under the unconstrained setting up to only constants. Then we can follow the proof of \citep[Theorem 4.1]{chen2024second} to conclude Theorem \ref{thm:LEN-restart}.
\end{proof}

Similarly, the result of the (restarted) A-LEN Algorithm \citep[Theorem 5.3]{chen2025computationally} also hold under the constrained setting, as stated below.

\begin{thm}[ALEN-restart]
Assume $h(\vz): \gZ \rightarrow \sR$ is $\mu$-uniformly convex and has $\rho$-Lipschitz continuous Hessians. 
There exists a second-order algorithm, specifically, the restart version of \citep[Algorithm 5.1]{chen2025computationally}, that reuses Hessians every $m$ iterations, and 
returns a point $\vz$ such that $\Vert \vz - \vz^* \Vert \le \epsilon$ in $\gO\left( m+ m^{5/7} (\rho/\mu)^{2/7} \log (d_0/ \epsilon) \log m  \right)$ lazy CRN oracle calls.
\end{thm}

\begin{proof}
As shown in Theorem \ref{thm:ANPE-restart}, the
restart scheme transforms the convergence under the convex setting to the uniformly convex setting in a black-box manner. Hence, we only need to show the convergence of A-LEN \citep[Theorem 5.3]{chen2025computationally} under the convex setting. The result for the unconstrained case can be found in  \citep[Theorem 5.3]{chen2025computationally}. Essentially, the same result also holds under the constrained setting we consider here. It is because \citep[Algorithm 5.1]{chen2025computationally} invokes the lazy CRN \citep[Algorithm 1]{doikov2023second} as a subroutine, while the proof of lazy CRN can be readily extended to the constrained case \citep[Appendix F]{doikov2023second}.
\end{proof}

\end{document}